\documentclass[10pt]{article}

\usepackage[backend=bibtex]{biblatex}
\usepackage{latexsym}
\usepackage{amssymb}
\usepackage{amsthm}
\usepackage{amscd}
\usepackage{amsmath}
\usepackage{mathrsfs}
\usepackage{graphicx}
\usepackage{hyperref}
\usepackage{shuffle}
\usepackage{ytableau}
\usepackage{stmaryrd}

\usepackage[all]{xy}
\input xy \xyoption{frame}
\xyoption{dvips}

\usepackage[colorinlistoftodos]{todonotes}
\usetikzlibrary{chains,scopes,decorations.markings}

\usepackage{mathdots}
\usepackage{tikz-cd}

\theoremstyle{definition}
\newtheorem* {theorem*}{Theorem}
\newtheorem* {conjecture*}{Conjecture}
\newtheorem{theorem}{Theorem}[section]

\theoremstyle{definition}

\newtheorem* {example*}{Example}

\newtheorem{lemma}[theorem]{Lemma}
\theoremstyle{definition}
\newtheorem{definition}[theorem]{Definition}
\theoremstyle{definition}

\newtheorem{proposition}[theorem]{Proposition}
\newtheorem{corollary}[theorem]{Corollary}

\newtheorem{remark}[theorem]{Remark}
\theoremstyle{definition}
\newtheorem {example}[theorem]{Example}
\theoremstyle{definition}

\theoremstyle{definition}

\theoremstyle{definition}

\theoremstyle{definition}

\newcommand{\ytabd}[1]{
\ytableausetup{boxsize = .8cm,aligntableaux=center}
{\small\begin{ytableau}  #1  \end{ytableau}}
}

\newcommand{\ytabb}[1]{
\ytableausetup{boxsize = .55cm,aligntableaux=center}
{\small\begin{ytableau}  #1  \end{ytableau}}
}

\newcommand{\ytab}[1]{
\ytableausetup{boxsize = .4cm,aligntableaux=center}
{\small\begin{ytableau}  #1  \end{ytableau}}
}

\def\({\left(}
\def\){\right)}

\newcommand{\FF}{\mathbb{F}}

\newcommand{\cC}{\mathcal{C}}
\newcommand{\cD}{\mathcal{D}}

\def\cX{\mathcal{X}}
\def\cW{\mathcal{W}}

\def\NN{\mathbb{N}}

\def\ZZ{\mathbb{Z}}

\def\ch{\mathrm{ch}}

\newcommand{\g}{\mathfrak{g}}

\def\barr{\begin{array}}
\def\earr{\end{array}}
\def\ba{\begin{aligned}}
\def\ea{\end{aligned}}
\def\be{\begin{equation}}
\def\ee{\end{equation}}

\def\quand{\quad\text{and}\quad}

\def\quord{\quad\text{or}\quad}

\newcommand{\gl}{\mathfrak{gl}}

\def\hs{\hspace{0.5mm}}

\def\id{\mathrm{id}}
\def\PP{\mathbb{P}}

\def\ben{\begin{enumerate}}
\def\een{\end{enumerate}}

\def\bei{\begin{itemize}}
\def\eei{\end{itemize}}

\def\cT{\mathscr{T}}

\def\hs{\hspace{0.5mm}}

\def\e{\textbf{e}}

\newcommand{\xRightarrow}[2][]{\ext@arrow 0359\Rightarrowfill@{#1}{#2}}

\newcommand{\rank}{\operatorname{rank}}

\renewcommand{\r}[1]{\textcolor{red}{#1}}

\newcommand{\cB}{\mathcal{B}}

\def\arcstart{\ \xy<0cm,-.06cm>\xymatrix@R=.1cm@C=.10cm }
\newcommand{\arcstartc}[1]{\ \xy<0cm,-.15cm>\xymatrix@R=.1cm@C=#1cm}

\def\r{\mathbf{r}}

\def\bar{\overline}

\definecolor{darkred}{rgb}{0.7,0,0} % darkred color
\newcommand{\defn}[1]{{\color{darkred}\emph{#1}}} % emphasis of a definition

\newcommand{\weight}{\mathsf{wt}}
\def\BB{\mathbb{B}}
\def\q{\mathfrak{q}}
\def\qq{\mathfrak{q}^+}

\def\DTab{\textsf{DecTab}}
\def\DDTab{\textsf{DecTab}^+}

\def\ShTab{\textsf{ShTab}}
\def\SShTab{\textsf{ShTab}^+}

\def\row{\mathsf{row}}
\def\revrow{\mathsf{revrow}}

\def\toQ{\xrightarrow{\hs\mathfrak{q}\hs}}
\def\toQQ{\xrightarrow{\hs\mathsf{dec}\hs}}

\def\Pqq{P_{\mathsf{dec}}}

\def\Qqq{Q_{\mathsf{dec}}}

\def\SD{\mathsf{SD}}

\def\unprime{\mathsf{unprime}}

\def\Thighest{T^{\mathsf{highest}}}
\def\Tlowest{T^{\mathsf{lowest}}}
\def\TTlowest{  \hat T^{\mathsf{lowest}}}

\def\height{\mathrm{height}}
\def\primes{\mathsf{primes}} 

\def\PHM{P_{\textsf{mix}}}

\usepackage{extarrows}

\def\cT{\mathcal{T}}
\def\simdec{\overset{\mathrm{dec}}\sim}

\def\r{\mathsf{reverse}}

\numberwithin{equation}{section}
\allowdisplaybreaks[1]
\UseCrayolaColors

\begin{document}
\title{Primed decomposition tableaux and extended queer crystals}
\author{
Eric MARBERG\thanks{
Department of Mathematics, HKUST, {\tt emarberg@ust.hk}.
}
\and
Kam Hung TONG\thanks{
Department of Mathematics, HKUST, {\tt khtongad@connect.ust.hk}.
}
}

\date{}

\maketitle

\begin{abstract}
Our previous work introduced a category of extended queer crystals, whose connected normal objects  have unique highest weight elements and   characters that are Schur $Q$-polynomials. The initial models for such crystals were based on semistandard shifted tableaux. Here, we introduce a simpler construction using certain ``primed'' decomposition tableaux, which slightly generalize the decomposition tableaux used in work of Grantcharov et al. This leads to a new, shorter proof of the highest weight properties of the normal subcategory of extended queer crystals. 
Along the way, we analyze a primed extension of Grantcharov et al.'s insertion scheme for decomposition tableaux.
\end{abstract}

\tableofcontents

\section{Introduction}

\defn{Crystals} are an abstraction for the crystal bases of quantum group representations, and can be viewed as acyclic directed graphs with labeled edges and weighted vertices, satisfying certain axioms. Crystals for $\gl_n$ and other classical Lie algebras were first studied by Kashiwara \cite{Kashiwara1990,Kashiwara1991} and Lusztig \cite{Lusztig1990a,Lusztig1990b} in the 1990s.
More recent work by Grantcharov et al. \cite{GJKKK15,GJKKK,MR2628823} introduced
crystals for the queer Lie superalgebra $\q_n$. 

Our previous article \cite{MT2023} defined a modified category of $\qq_n$-crystals, which share many nice features with $\gl_n$-crystals and $\q_n$-crystals. For example,
$\qq_n$-crystals have a natural tensor product and 
a standard crystal corresponding to the vector representation of the quantum group $U_q(\q_n)$. 
This lets one define a subcategory of \defn{normal crystals}, consisting of crystals whose connected components can each be embedded in some tensor power of the standard crystal.

In \cite{MT2023}, our primary models for normal $\qq_n$-crystals were derived from \defn{semistandard shifted tableaux}, using fairly technical crystal operators introduced in \cite{AssafOguz,HPS,Hiroshima2018}.
One of the main results of this paper is to introduce a much simpler model for
normal $\qq_n$-crystal based on a ``primed'' generalization of \defn{decomposition tableaux}.
The latter tableaux served as the original model for normal (non-extended) $\q_n$-crystals in \cite{GJKKK}.

After formally defining primed decomposition tableaux, we equip them with a family of $\qq_n$-crystal operators, identify their highest weight elements, and construct a primed generalization of a useful ``insertion scheme'' from \cite{GJKKK}, which we refer to as \defn{decomposition insertion}.
As an application, we give a short, alternate proof that normal $\qq_n$-crystals are determined up to isomorphism by their characters (which range over all Schur $Q$-positive symmetric polynomials in $n$ variables), and also by their multisets of highest weights (which range over all strict partitions with at most $n$ parts).
These results are explained in Section~\ref{normal-sect}.

We also derive a simpler, alternate characterization of highest and lowest weight elements for normal $\q_n$ and $\qq_n$-crystals (see Proposition~\ref{ranked-prop}).
 As another application, we identify in Section~\ref{comp-sect}
 the equivalence relation on primed words whose classes share the same
 decomposition insertion tableau.

\subsection*{Acknowledgments}

This work was supported by Hong Kong RGC grants 16306120 and 16304122. 
We thank Joel Lewis and Travis Scrimshaw for useful discussions.

\section{Preliminaries}

Let $\NN =\{0,1,2,\dots\}$ and $\PP=\{1,2,3,\dots\}$.
Fix $ n \in \NN$ and let $[n] = \{1,2,\dots,n\}$.
Throughout, let $\e_1,\e_2,\dots,\e_n\in\ZZ^n$ be the standard basis.

\subsection{Shifted tableaux}

Assume $\lambda = (\lambda_1 > \lambda_2 > \dots > 0)$ is a \defn{strict} partition. Let $\ell(\lambda)$ be the number of nonzero parts of $\lambda$.
The \defn{shifted diagram} of $\lambda$ is 
the set  
\[\SD_\lambda:= \{ (i,i+j-1) : i\in  [\ell(\lambda)] \text{ and } j \in [\lambda_i]\}.\]
A \defn{shifted tableau} of shape $\lambda$  is a map $ \SD_\lambda \to \{1'<1<2'<2<\dots\}$.

If $T$ is a shifted tableau, then we  write $(i,j) \in T$ to indicate that $(i,j)$ belong to the domain of $T$
and we let $T_{ij}$ denote the value assigned to this position.
We draw tableaux in French notation, so that row indices increase from bottom to top and column indices increase from left to right.
Both
\be\label{tableau-ex}
S = \ytab{  \none& 3 & 5& 7 \\ 1 & 2 & 4 & 6}
\quand 
T = \ytab{ \none & 2' & 2 & 4' \\ 1' & 1 & 1 & 4'}  
\ee
are shifted tableaux of shape $\lambda=(4,3)$
with $S_{23} = 5$ and $T_{23} = 2$. The \defn{(main) diagonal} of a shifted tableau
is the set of boxes $(i,j)$ in its domain with $i=j$.

A shifted tableau is \defn{semistandard} if its  rows and columns are weakly increasing,
such that no primed number appears more than once in any row and no unprimed number appears more than one in any column.
The examples 
in \eqref{tableau-ex}
are both semistandard.
We write  
  $\SShTab(\lambda)$
 for the set of all semistandard shifted tableaux of shape $\lambda$,
 and $ \ShTab(\lambda)$ for the subset of elements in $ \SShTab(\lambda)$
with no primed entries on the diagonal.
   Define  $ \ShTab_n(\lambda)\subseteq \ShTab(\lambda)$ and $ \SShTab_n(\lambda)\subseteq \SShTab(\lambda)$ to be the subsets of shifted tableaux with all entries at most $n$.

Our main references below are \cite[Chapter III, \S8]{Macdonald} and \cite[\S3.3]{IkedaNaruse}.
If $T$ is a shifted tableau, then   set
$ x^T := x_1^{a_1} x_2^{a_2} \cdots x_n^{a_n}$ where $a_k$ is the number of times $k$ or $k'$ appears in $T$.
The \defn{Schur $P$- and $Q$-functions}  of a strict partition $\lambda$ are  
\be\label{PQ-eq}
P_\lambda := \sum_{T \in \ShTab(\lambda)} x^T
\quand
Q_\lambda := \sum_{T \in \SShTab(\lambda)} x^T = 2^{\ell(\lambda)} P_\lambda .\ee
These power series are both of bounded degree and symmetric in the $x_i$ variables. 
% As $\lambda$ varies over all strict partitions in $\NN^n$,
%the sets $\{Q_\lambda\}$ and $\{P_\lambda\}$ are each $\ZZ$-bases for subrings of $ \Sym$.
%$
%\SymQ\subset \SymP\subset \Sym
%$.
% characterized as 
%\[\ba
%\SymP&= \left\{ f \in \Sym: f(x_1,-x_1,x_3,x_4,\dots) \in \ZZ\llbracket x_3,x_4,\dots\rrbracket\right\},
%\\ 
%\SymQ &=\left \{ f \in \SymP : f - f(0,x_2,x_3,\dots) \in 2x_1\ZZ\llbracket x_1,x_2,\dots\rrbracket\right\}.
%\ea\]

We write $P_\lambda(x_1,x_2,\dots,x_n)$ and $Q_\lambda(x_1,x_2,\dots,x_n)$ for the   polynomials obtained by specializing $P_\lambda$ and $Q_\lambda$ to $n$ variables,
or equivalently by taking the finite sums 
$\sum_{T \in \ShTab_n(\lambda)} x^T$
and
$ \sum_{T \in \SShTab_n(\lambda)} x^T$. 
As $\lambda$ varies over all strict partitions with $\ell(\lambda)\leq n$,
these polynomials 
are linearly independent over $\ZZ$.

\subsection{Abstract crystals}

Let $\cB$ be a set with maps $\weight  :  \cB\to \ZZ^n$
and 
$e_i,f_i  :  \cB \to \cB \sqcup \{0\}$  for $i \in [n-1]$,
where $0 \notin \cB$ is an auxiliary element.

\begin{definition}
\label{crystal-def}
The set $\cB$ is a
 \defn{$\gl_n$-crystal}
if 
for all $i \in [n-1]$ and $b,c \in \cB$
 it holds that  $e_i(b) = c $ if and only if $ f_i(c) = b$, 
in which case $ \weight(c) = \weight(b) + \e_{i} -\e_{i+1}$.
\end{definition}

Assume $\cB$ is a $\gl_n$-crystal. Then the maps $e_i$ and $f_i$ encode a directed graph with vertex set $\cB$, to be called the \defn{crystal graph},
with an edge $b \xrightarrow{i} c$ if and only if $ f_i(b)=c$.
Define the \defn{string lengths} $\varepsilon_i, \varphi_i  :  \cB \to \{0,1,2,\dots\}\sqcup \{\infty\}$ 
by \be\label{string-eqs}
\varepsilon_i(b) := \sup\left\{ k\geq 0 : e_i^k(b) \neq 0\right\}
\text{ and }
\varphi_i(b) := \sup\left\{ k \geq 0: f_i^k(b) \neq 0\right\}.
\ee

\begin{definition} The $\gl_n$-crystal $\cB$ is \defn{seminormal}
if the string lengths take only finite values with $\varphi_i(b) - \varepsilon_i(b) = \weight(b)_i - \weight(b)_{i+1}$ for all $i\in[n-1]$ and $b \in \cB$.
\end{definition}

If $\cB$ is finite then its \defn{character} is the Laurent polynomial 
\be\textstyle \ch(\cB) := \sum_{b \in \cB} x^{\weight(b)}\quad\text{where }x^{\weight(b)} := \prod_{i\in[n]} x_i^{\weight(b)_i}.\ee
The character is symmetric in $x_1,x_2,\dots,x_n$ if $\cB$ is seminormal \cite[\S2.6]{BumpSchilling}.

We refer to $\weight$ as the \defn{weight map}, to each $e_i$ as a \defn{raising operator}, and to each $f_i$ as a \defn{lowering operator}.
Each connected component of the crystal graph of $\cB$ may be viewed as a $\gl_n$-crystal
by restricting the weight map and crystal operators; these objects are called \defn{full subcrystals}.

\begin{example}\label{st-ex1}
The \defn{standard $\gl_n$-crystal} $\BB_n = \left\{ \boxed{i}: i \in [n]\right\}$ has
crystal graph
\[
    \begin{tikzpicture}[xscale=1.8, yscale=1,>=latex,baseline=(z.base)]
    \node at (0,0.0) (z) {};
      \node at (0,0) (T0) {$\boxed{1}$};
      \node at (1,0) (T1) {$\boxed{2}$};
      \node at (2,0) (T2) {$\boxed{3}$};
      \node at (3,0) (T3) {${\cdots}$};
      \node at (4,0) (T4) {$\boxed{n}$};
      \draw[->,thick]  (T0) -- (T1) node[midway,above,scale=0.75] {$1$};
      \draw[->,thick]  (T1) -- (T2) node[midway,above,scale=0.75] {$2$};
      \draw[->,thick]  (T2) -- (T3) node[midway,above,scale=0.75] {$3$};
      \draw[->,thick]  (T3) -- (T4) node[midway,above,scale=0.75] {$n-1$};
     \end{tikzpicture}
\quad\text{with }\weight(\boxed{i}):=\e_i.
\]
  \end{example}
  
  If $\cB$ and $\cC$ are $\gl$-crystals
then the set   $\cB \otimes \cC := \{ b\otimes c : b\in \cB,\ c\in \cC\}$
 of formal tensors has a unique $\gl_n$-crystal structure
 (which is seminormal if $\cB$ and $\cC$ are seminormal)
in which
$
\weight(b\otimes c) := \weight(b) + \weight(c)
$
and  
\be
e_i(b\otimes c) := \begin{cases}
b \otimes e_i(c) &\text{if }\varepsilon_i(b) \leq \varphi_i(c) \\
e_i(b) \otimes c &\text{if }\varepsilon_i(b) > \varphi_i(c)
\end{cases}
\ee
and
\be
f_i(b\otimes c) := \begin{cases}
b \otimes f_i(c) &\text{if }\varepsilon_i(b) < \varphi_i(c) \\
f_i(b) \otimes c &\text{if }\varepsilon_i(b) \geq \varphi_i(c)
\end{cases}
\ee
for $i \in [n-1]$,
where we set $b\otimes 0 = 0\otimes c = 0$  \cite[\S2.3]{BumpSchilling}.
This follows the ``anti-Kashiwara convention,''
which reverses the tensor product order  in \cite{GJKKK15,GJKKK}.
The natural map $\cB \otimes (\cC \otimes \cD) \to (\cB \otimes \cC) \otimes \cD$ is a crystal isomorphism,
 so we can dispense with   parentheses in iterated tensor products.

\subsection{Queer crystals}

The general linear Lie algebra $\gl_n$ has two super-analogues, one of which is the \defn{queer Lie superalgebra} $\q_n$. Grantcharov et al. developed a theory of crystals for $\q_n$ in \cite{GJKKK15,GJKKK,MR2628823}, which we review here. Assume $n\geq 2$.
  
Let $\cB$ be a $\gl_n$-crystal with maps $e_{\bar 1},f_{\bar 1} : \cB \to \cB \sqcup \{0\}$.
Define $\varepsilon_{ \bar 1}, \varphi_{ \bar 1} : \cB \to \NN\sqcup \{\infty\}$ 
as in  \eqref{string-eqs}   with $i=\bar 1$.
Below, we say that one map $\phi : \cB \to \cB\sqcup \{0\}$ \defn{preserves} another map $\eta : \cB \to \cX$
if $\eta(\phi(b)) = \eta(b)$ whenever $\phi(b) \neq 0$.

\begin{definition}
\label{q-crystal-def}
The $\gl_n$-crystal $\cB$ is a
 \defn{$\q_n$-crystal} 
  if both of the following hold:
\ben
\item[(a)]  $e_{\bar 1}$, $f_{\bar 1}$ commute with $e_i$, $f_i$ while preserving $\varepsilon_i$, $\varphi_i$ for all $3\leq i \leq n-1$;
\item[(b)] if $b,c \in \cB$ then
$e_{\bar 1}(b) = c$ if and only if $f_{\bar 1}(c) = b,$
in which case
 \[ \weight(c) = \weight(b) + \e_{1} -\e_{2}.\]
 \een
\end{definition}

Assume $\cB$ is a $\q_n$-crystal. The corresponding \defn{$\q_n$-crystal graph}  has vertex set $\cB$ and edges $b \xrightarrow{i} c$ whenever $f_i(b) =c$ for any $i \in \{\bar 1,1,2,\dots,n-1\}$.

\begin{definition} 
A $\q_n$-crystal $\cB$ is \defn{seminormal}
if it is seminormal as a $\gl_n$-crystal and for all $b \in \cB$ one has both $\weight(b)\in \NN^n$ and 
\[
\varphi_{\bar 1}(b) + \varepsilon_{\bar 1}(b)=\begin{cases} 0 &\text{if }\weight(b)_1 =\weight(b)_{2}= 0 \\ 1&\text{otherwise}.\end{cases}\]
\end{definition}

If $\cB$ is a finite seminormal $\q_n$-crystal then $\ch(\cB)$ is a $\ZZ$-linear combination of Schur $P$-polynomials $P_\lambda(x_1,x_2,\dots,x_n)$
by \cite[Prop. 2.5]{Marberg2019b}.

\begin{example}\label{st-ex2}
The \defn{standard $\q_n$-crystal} $\BB_n = \left\{ \boxed{i}: i \in [n]\right\}$ has
crystal graph
\[
    \begin{tikzpicture}[xscale=1.8, yscale=1,>=latex,baseline=(z.base)]
    \node at (0,0.0) (z) {};
      \node at (0,0) (T0) {$\boxed{1}$};
      \node at (1,0) (T1) {$\boxed{2}$};
      \node at (2,0) (T2) {$\boxed{3}$};
      \node at (3,0) (T3) {${\cdots}$};
      \node at (4,0) (T4) {$\boxed{n}$};
      \draw[->,darkred,thick]  (T0.15) -- (T1.165) node[midway,above,scale=0.75] {$ \overline1$};
      \draw[->,thick]  (T0) -- (T1) node[midway,below,scale=0.75] {$ 1$};
      \draw[->,thick]  (T1) -- (T2) node[midway,below,scale=0.75] {$ 2$};
      \draw[->,thick]  (T2) -- (T3) node[midway,below,scale=0.75] {$ 3$};
      \draw[->,thick]  (T3) -- (T4) node[midway,below,scale=0.75] {$ {n-1}$};
     \end{tikzpicture}
  \quad\text{with }\weight(\boxed{i}):=\e_i.
  \]
\end{example}

Suppose $\cB$ and $\cC$ are $\q_n$-crystals.
The set 
$\cB \otimes \cC$ already has a $\gl_n$-crystal structure. 
There is a unique way of viewing this object as a $\q_n$-crystal
with
\be\label{q-o-1}
  e_{\overline 1}(b\otimes c) := \begin{cases} 
 b \otimes e_{\overline 1}(c)&\text{if }e_{\overline1}(b)=f_{\overline1}(b)=0
 \\
  e_{\overline 1}(b) \otimes c
&\text{otherwise}
 \end{cases}
\ee
and
\be\label{q-o-2}
  f_{\overline 1}(b\otimes c) := \begin{cases} 
 b \otimes f_{\overline 1}(c)&\text{if }e_{\overline1}(b)=f_{\overline1}(b)=0
 \\
  f_{\overline 1}(b) \otimes c
&\text{otherwise}
 \end{cases}
\ee
where it is again understood that $b\otimes 0 = 0\otimes c = 0$ \cite[Thm.~1.8]{GJKKK}.
As in the $\gl_n$-case, 
the natural map $\cB \otimes (\cC \otimes \cD) \to (\cB \otimes \cC) \otimes \cD$ is a crystal isomorphism,
and if $\cB$ and $\cC$ are seminormal then so is $\cB\otimes \cC$.

\subsection{Extended crystals}

We continue to assume $n\geq 2$. The following theory of \defn{extended $\q_n$-crystals}
(abbreviated as \defn{$\qq_n$-crystals} from now on) was introduced in our previous work \cite{MT2023}.
Suppose $\cB$ is a $\q_n$-crystal with additional maps $e_{0},f_{0}: \cB \to \cB\sqcup\{0\}$.
Define   $\varepsilon_{0},\varphi_{0} : \cB \to \NN\sqcup\{\infty\}$
by the formula \eqref{string-eqs} with $i=0$.

\begin{definition}
\label{qq-crystal-def}
The $\q_n$-crystal $\cB$ is a
 \defn{$\qq_n$-crystal} 
  if  the following all hold: \ben
  \item[(a)] 
  the operators $e_{0}$ and $f_{0}$ commute with $e_i$ and $f_i$ for   $2\leq i \leq n-1$
  while preserving both  $\weight$ and the string lengths
      $\varepsilon_i$ and $\varphi_i$ for all $i\neq 0$;

\item[(b)] if $b,c \in \cB$ then
$e_{0}(b) = c $ if and only if $ f_{0}(c) = b$; and 

      \item[(c)] if $b \in \cB$ then $\varepsilon_0(b) + \varphi_0(b) \leq 1$,
      with $\varepsilon_0(b) + \varphi_0(b) =0$ if $\varepsilon_{\overline1}(b) + \varphi_{\overline1}(b) =0$.
\een
\end{definition}

Assume $\cB$ is a $\qq_n$-crystal. We have $\weight(e_0(b)) = \weight(b)$ for $b \in \cB$ with $e_0(b)\neq 0$ since $e_0$ preserves the weight map.
However, it always holds that $e_0(b) \neq b$ and $f_0(b) \neq b$ 
since $\varepsilon_0(b) + \varphi_0(b) \leq 1$.
The   \defn{$\qq_n$-crystal graph} of $\cB$ has vertex set $\cB$ and edges $b \xrightarrow{i} c$ whenever $f_i(b) =c$ for any $i \in \{\bar 1,0,1,2,\dots,n-1\}$.

\begin{definition}
The $\qq_n$-crystal $\cB$ is \defn{seminormal} if it is seminormal as a $\q_n$-crystal 
and for all $b \in \cB$ it holds that $\varphi_{0}(b) + \varepsilon_{0}(b)=\begin{cases} 
0&\text{if $\weight(b)_1=0$} \\
1&\text{if $\weight(b)_1>0$}.\end{cases}$
\end{definition}

If $\cB$ is a finite seminormal $\qq_n$-crystal then $\ch(B)$ is a $\ZZ$-linear combination of Schur $Q$-polynomials $Q_\lambda(x_1,x_2,\dots,x_n)$
by \cite[Prop. 3.13]{MT2023}.

\begin{example}\label{st-ex3}
The \defn{standard $\qq_n$-crystal} $\BB^+_n$ has
crystal graph
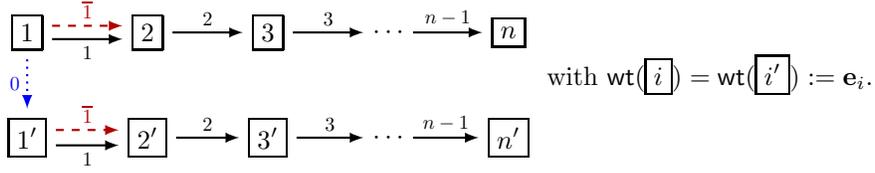
\[
      \begin{tikzpicture}[xscale=1.6, yscale=1,>=latex,baseline=(z.base)]
    \node at (0,0.7) (z) {};
      \node at (0,0) (T0) {$\boxed{1'}$};
      \node at (1,0) (T1) {$\boxed{2'}$};
      \node at (2,0) (T2) {$\boxed{3'}$};
      \node at (3,0) (T3) {${\cdots}$};
      \node at (4,0) (T4) {$\boxed{n'}$};
      \node at (0,1.4) (U0) {$\boxed{1}$};
      \node at (1,1.4) (U1) {$\boxed{2}$};
      \node at (2,1.4) (U2) {$\boxed{3}$};
      \node at (3,1.4) (U3) {${\cdots}$};
      \node at (4,1.4) (U4) {$\boxed{n}$};
      \draw[->,thick,dashed,color=darkred]  (T0.15) -- (T1.165) node[midway,above,scale=0.75] {$\overline 1$};
      \draw[->,thick]  (T0.345) -- (T1.195) node[midway,below,scale=0.75] {$1$};
      \draw[->,thick]  (T1) -- (T2) node[midway,above,scale=0.75] {$2$};
      \draw[->,thick]  (T2) -- (T3) node[midway,above,scale=0.75] {$3$};
      \draw[->,thick]  (T3) -- (T4) node[midway,above,scale=0.75] {$n-1$};
      \draw[->,thick,dashed,color=darkred]  (U0.15) -- (U1.165) node[midway,above,scale=0.75] {$\overline 1$};
      \draw[->,thick]  (U0.345) -- (U1.195) node[midway,below,scale=0.75] {$1$};
      \draw[->,thick]  (U1) -- (U2) node[midway,above,scale=0.75] {$2$};
      \draw[->,thick]  (U2) -- (U3) node[midway,above,scale=0.75] {$3$};
      \draw[->,thick]  (U3) -- (U4) node[midway,above,scale=0.75] {$n-1$};
      \draw[->,thick,dotted,color=blue]  (U0) -- (T0) node[midway,left,scale=0.75] {$0$};
     \end{tikzpicture}
  \text{ with }  \weight(\boxed{i})=\weight(\boxed{i'}):=\e_i.
 \]
\end{example}

Suppose $\cB$ and $\cC$ are $\qq_n$-crystals.
The $\gl_n$-crystal 
$\cB \otimes \cC$ has a unique $\qq_n$-crystal structure
(which is seminormal if $\cB$ and $\cC$ are seminormal)
with
\be
 e_{0}(b\otimes c) := \begin{cases} 
 b \otimes e_0(c) & \text{if $e_{0}(b)=f_{0}(b)=0$}
 \\
 e_0(b) \otimes c&\text{otherwise}
 \end{cases}
\ee
and
\be
  f_{0}(b\otimes c) := \begin{cases} 
 b \otimes f_0(c) & \text{if $e_{0}(b)=f_{0}(b)=0$}
 \\
 f_0(b) \otimes c&\text{otherwise}
 \end{cases}
\ee
along with
\be\label{eb1-eq}
  e_{\overline 1}(b\otimes c) := \begin{cases} 
 b \otimes e_{\overline 1}(c)
 &\text{if $e_{\overline1}(b)=f_{\overline1}(b)=0$}
 \\[-12pt]
 \\
 f_0  e_{\bar 1} (b) \otimes e_0(c) 
&\text{if $f_0e_{\overline1}(b)\neq0 \neq e_0(c)$} \\ &\text{and  $e_0 (b)=f_0(b)=0$}
\\[-12pt]
\\
 e_0  e_{\bar 1} (b) \otimes f_0(c) 
&\text{if $e_0e_{\overline1}(b)\neq0 \neq f_0(c)$} \\ &\text{and $ e_0 (b)=f_0(b)=0$}
\\[-12pt]
 \\
  e_{\overline 1}(b) \otimes c
&\text{otherwise}
 \end{cases}
\ee
and
\be\label{fb1-eq}
 f_{\overline 1}(b\otimes c) := \begin{cases} 
 b \otimes f_{\overline 1}(c)&\text{if $e_{\overline1}(b)=f_{\overline1}(b)=0$} 
 \\[-12pt]
 \\
  f_{\overline 1}e_0(b) \otimes f_0(c)
 &\text{if $f_{\overline1}e_0(b)\neq0 \neq f_0(c)$} \\ &\text{and $e_0f_{\overline1}e_0(b)=f_0f_{\overline1}e_0(b)=0$}
\\[-12pt]
\\
  f_{\overline 1}f_0(b) \otimes e_0(c)
 &\text{if $f_{\overline1}f_0(b)\neq0 \neq e_0(c)$} \\ &\text{and $e_0f_{\overline1}f_0(b)=f_0f_{\overline1}f_0(b)=0$}
\\[-12pt]
\\
    f_{\overline 1}(b) \otimes c &\text{otherwise}
 \end{cases}
 \ee
% \be
% e_{0}(b\otimes c) := \begin{cases} 
% e_0(b) \otimes c&\text{if $\weight(b)_1 \neq 0$}
% \\
%b \otimes e_0(c) & \text{if $\weight(b)_1  = 0$}
% \end{cases}
%\ee
%and
%\be
%  f_{0}(b\otimes c) := \begin{cases} 
% f_0(b) \otimes c&\text{if $\weight(b)_1 \neq 0$}
% \\
%b \otimes f_0(c) & \text{if $\weight(b)_1  = 0$}
% \end{cases}
%\ee
%along with
%\be
%  e_{\overline 1}(b\otimes c) := \begin{cases} 
% b \otimes e_{\overline 1}(c)
% &\text{if $\weight(b)_1 = \weight(b)_2 = 0$}
% \\
% f_0  e_{\bar 1} (b) \otimes e_0(c) 
% &\text{if $\weight(b)_1 = 0$ and $f_0e_{\bar 1}(b) \neq 0\neq e_0(c)$}
%\\
% e_0  e_{\bar 1} (b) \otimes f_0(c) 
% &\text{if $\weight(b)_1 = 0$ and $e_0e_{\bar 1}(b)\neq 0\neq f_0(c) $}
% \\
%  e_{\overline 1}(b) \otimes c
%&\text{otherwise}
% \end{cases}
%\ee
%and
%\be
% f_{\overline 1}(b\otimes c) := \begin{cases} 
% b \otimes f_{\overline 1}(c)&\text{if $\weight(b)_1 = \weight(b)_2 = 0$}
% \\
%  f_{\overline 1}f_0(b) \otimes e_0(c)
% &\text{if $\weight(b)_1 = 1$ and $f_{\overline 1}f_0(b)\neq 0 \neq e_0(c)$}
%\\
%  f_{\overline 1}e_0(b) \otimes f_0(c)
% &\text{if $\weight(b)_1 = 1$ and $f_{\overline 1}e_0(b) \neq 0\neq f_0(c)$}
%  \\
%    f_{\overline 1}(b) \otimes c &\text{otherwise}
% \end{cases}
% \ee
where again one sets $b\otimes 0 = 0\otimes c = 0$ \cite[Thm.~3.14]{MT2023}.

%\begin{remark}
%The fourth case of \eqref{eb1-eq}
%when $e_{\overline 1}(b\otimes c) = e_{\overline1}(b)\otimes c$
%occurs if and only if
%it holds that 
%$e_{\bar 1}(b) \neq 0 $ or $f_{\bar 1}(b) \neq 0$,
%and if $e_0(b) = f_0(b)=0$ then we have both
%(1) $f_0e_{\bar 1}(b) =0$ or $e_0(c)= 0$,
%and also 
%(2) $e_0e_{\bar 1}(b) =0$ or $f_0(c)= 0$.
%\end{remark}

%\begin{remark}
%How we evaluate $e_{\bar 1}(b\otimes c)$ and $f_{\bar 1}(b\otimes c)$,
%depends on whether $\cB$ and $\cC$ are viewed as
%$\q_n$- or $\qq_n$-crystals. For this reason, using ``$\otimes$'' for both the $\q_n$- and $\qq_n$-crystal tensor product
%is potentially ambiguous.
%However, we expect that this convention will not cause much confusion in practice.
%\end{remark}

\begin{remark}
When $\cB$ and $\cC$ are seminormal $\qq_n$-crystals, 
the definitions  $e_{\bar 1} $ and $f_{\bar 1}$ just given simplify
to the following formulas from
\cite[Thm.~3.14]{MT2023}:
\[
  e_{\overline 1}(b\otimes c) = \begin{cases} 
 b \otimes e_{\overline 1}(c)
 &\text{if $\weight(b)_1 = \weight(b)_2 = 0$}
 \\
 f_0  e_{\bar 1} (b) \otimes e_0(c) 
 &\text{if $\weight(b)_1 = 0$ and $f_0e_{\bar 1}(b) \neq 0\neq e_0(c)$}
\\
 e_0  e_{\bar 1} (b) \otimes f_0(c) 
 &\text{if $\weight(b)_1 = 0$ and $e_0e_{\bar 1}(b)\neq 0\neq f_0(c) $}
 \\
  e_{\overline 1}(b) \otimes c
&\text{otherwise}
 \end{cases}
\]
and
\[
 f_{\overline 1}(b\otimes c) = \begin{cases} 
 b \otimes f_{\overline 1}(c)&\text{if $\weight(b)_1 = \weight(b)_2 = 0$}
 \\
  f_{\overline 1}f_0(b) \otimes e_0(c)
 &\text{if $\weight(b)_1 = 1$ and $f_{\overline 1}f_0(b)\neq 0 \neq e_0(c)$}
\\
  f_{\overline 1}e_0(b) \otimes f_0(c)
 &\text{if $\weight(b)_1 = 1$ and $f_{\overline 1}e_0(b) \neq 0\neq f_0(c)$}
  \\
    f_{\overline 1}(b) \otimes c &\text{otherwise}.
 \end{cases}
 \]
We also mention
that how one evaluates $e_{\bar 1}(b\otimes c)$ and $f_{\bar 1}(b\otimes c)$,
depends on whether $\cB$ and $\cC$ are viewed as
$\q_n$- or $\qq_n$-crystals,
since the formulas \eqref{eb1-eq} and \eqref{fb1-eq} do not agree with \eqref{q-o-1} and \eqref{q-o-2}.
 For this reason, using ``$\otimes$'' for both the $\q_n$- and $\qq_n$-crystal tensor product
is ambiguous.
However, we expect that this convention will not cause much confusion in practice.
 \end{remark}

 We include a proof of the following result for completeness, since it was only shown
in the seminormal case in \cite{MT2023}.

\begin{proposition}
If $\cB$, $\cC$, and $\cD$ are $\qq_n$-crystals then the bijection
$(\cB\otimes\cC) \otimes \cD \to \cB\otimes(\cC\otimes \cD)$ 
given by $(b\otimes c)\otimes d \mapsto b\otimes(c\otimes d)$
is a $\qq_n$-crystal isomorphism.
\end{proposition}

\begin{proof}
The natural map
$(\cB\otimes\cC) \otimes \cD \to \cB\otimes(\cC\otimes \cD)$
commutes with the $\gl_n$-crystal operators and also with $e_0$ and $f_0$, while preserving the weight map.
It remains to check that this map commutes with $e_{\bar 1}$ and $f_{\bar 1}$. 
Because the map is a bijection, it suffices check that it commutes with just $e_{\bar 1}$.
Fix $b \in\cB$, $c \in \cC$, and $d \in \cD$.
We check that $e_{\bar 1}(b \otimes (c\otimes d)) = e_{\bar 1}((b \otimes c)\otimes d)$:
\begin{itemize}
\item[(a)] Assume that $e_{\bar1}(b)=f_{\bar1}(b) = 0$.
If $e_{\bar1}(c) =f_{\bar1}(c) = 0$, then 
$
 e_{\bar1}(b\otimes c) =0$
 so
\[
e_{\bar 1}(b \otimes (c\otimes d)) 
=b\otimes e_{\bar 1}(c\otimes d)
=b\otimes c\otimes e_{\bar 1}(d)
= e_{\bar 1}((b \otimes c)\otimes d).
\]
Since $e_{0}(b)=f_{0}(b) = 0$ by condition (c) in Definition~\ref{qq-crystal-def}, it holds that
\[
f_0e_{\bar 1}(b\otimes c)= b\otimes f_0e_{\bar 1}(c) 
\quand
e_0e_{\bar 1}(b\otimes c)= b\otimes e_0e_{\bar 1}(c).
\]
Thus 
\[\ba 
 f_0e_{\overline1}(b\otimes c) \neq 0 &\text{\ \ if and only if\ \ } f_0e_{\overline1}(c)\neq 0,
 \\
e_0e_{\overline1}(b\otimes c) \neq 0 &\text{\ \  if and only if\ \  }e_0e_{\overline1}(c)\neq 0,
 \text{ and }
 \\
 e_0(b\otimes c)= f_0(b\otimes c) = 0 &\text{\ \  if and only if\ \  }  e_0(c)= f_0(c) = 0.
 \ea
 \]
Therefore if $f_0e_{\overline1}(c)\neq0 \neq e_0(d)$ and $ e_0 (c)=f_0(c)=0$, then 
\[\ba
e_{\bar 1}(b \otimes (c\otimes d)) 
&=b\otimes e_{\bar 1}(c\otimes d)\\
&=b\otimes f_0 e_{\bar 1}(c)\otimes e_{0}(d)\\
&=f_0e_{\bar 1}(b \otimes c)\otimes e_0(d)
= e_{\bar 1}((b \otimes c)\otimes d),\ea
\]
while if 
 $e_0e_{\overline1}(c)\neq0 \neq f_0(d)$ and $ e_0 (c)=f_0(c)=0$, then 
\[\ba
e_{\bar 1}(b \otimes (c\otimes d)) 
&=b\otimes e_{\bar 1}(c\otimes d)\\
&=b\otimes e_0 e_{\bar 1}(c)\otimes f_{0}(d)\\
&=e_0e_{\bar 1}(b \otimes c)\otimes f_0(d)
= e_{\bar 1}((b \otimes c)\otimes d),\ea
\]
while in the remaining case 
\[\ba
e_{\bar 1}(b \otimes (c\otimes d)) 
&=b\otimes e_{\bar 1}(c\otimes d)\\
&=b\otimes   e_{\bar 1}(c)\otimes d\\
&= e_{\bar 1}(b \otimes c)\otimes d
= e_{\bar 1}((b \otimes c)\otimes d).\ea
\]

\item[(b)] Next assume that
$f_0e_{\overline1}(b)\neq0 \neq e_0(c\otimes d)$ and $ e_0 (b)=f_0(b)=0$.
In this case $e_{\bar 1}(b\otimes c)\neq 0$.
If  $e_0(c\otimes d) = e_0(c) \otimes d$, then 
$e_0(c) \neq 0$
so
\[
e_{\bar 1}(b \otimes (c\otimes d)) =f_0e_{\bar 1}(b)\otimes e_0(c)\otimes d
=
e_{\bar 1}(b \otimes c)\otimes d
=
 e_{\bar 1}((b \otimes c)\otimes d).
\]
If $e_0(c\otimes d) = c\otimes e_0(d)$, then  $e_0(c)=f_0(c)=0$ and $e_0(d) \neq 0$.
This means 
\[
f_0e_{\bar 1}(b\otimes c) = f_0(e_{\bar 1}(b)\otimes c) = f_0 e_{\bar 1}(b)\otimes c \neq 0\]
and
$e_0(b\otimes c) = f_0(b\otimes c) = 0$, so 
\[
e_{\bar 1}(b \otimes (c\otimes d)) 
=
f_0e_{\bar 1}(b)\otimes c\otimes e_0(d)
=
f_0e_{\bar 1}(b\otimes c)\otimes e_0(d)
=
 e_{\bar 1}((b \otimes c)\otimes d).
\]

\item[(c)] Now assume that
$e_0e_{\overline1}(b)\neq0 \neq f_0(c\otimes d)$ and $ e_0 (b)=f_0(b)=0$.
This case is almost the same as the previous one.
We again have $e_{\bar 1}(b\otimes c)\neq 0$.
If  $f_0(c\otimes d) = f_0(c) \otimes d$, then 
$f_0(c) \neq 0$
so
\[
e_{\bar 1}(b \otimes (c\otimes d)) =e_0e_{\bar 1}(b)\otimes f_0(c)\otimes d
=
e_{\bar 1}(b \otimes c)\otimes d
=
 e_{\bar 1}((b \otimes c)\otimes d).
\]
If $f_0(c\otimes d) = c\otimes f_0(d)$, then  $e_0(c)=f_0(c)=0$ and $f_0(d) \neq 0$.
This means 
\[
e_0e_{\bar 1}(b\otimes c) = e_0(e_{\bar 1}(b)\otimes c) = e_0 e_{\bar 1}(b)\otimes c \neq 0\]
and
$e_0(b\otimes c) = f_0(b\otimes c) = 0$, so 
\[
e_{\bar 1}(b \otimes (c\otimes d)) 
=
e_0e_{\bar 1}(b)\otimes c\otimes f_0(d)
=
e_0e_{\bar 1}(b\otimes c)\otimes f_0(d)
=
 e_{\bar 1}((b \otimes c)\otimes d).
\]

\item[(d)] Finally suppose $e_{\bar 1}(b) \neq 0 $ or $f_{\bar 1}(b) \neq 0$,
and that if $e_0(b) = f_0(b)=0$ then 
\ben
\item[(1)] $f_0e_{\bar 1}(b) =0$ or $e_0(c\otimes d)= 0$,
and also 
\item[(2)]  $e_0e_{\bar 1}(b) =0$ or $f_0(c\otimes d)= 0$.
\een
This is precisely the last case in our definition of $e_{\bar 1}$, which gives 
\[
e_{\bar 1}(b \otimes (c\otimes d)) =
  e_{\bar 1}(b) \otimes c\otimes d .
  \]
  Since  $e_0(c\otimes d)= 0$ implies that $e_0(c)= 0$ and since 
  $f_0(c\otimes d)= 0$ implies that $f_0(c)= 0$,
  we also have
\[
e_{\bar 1}(b \otimes c) =
  e_{\bar 1}(b) \otimes c.
  \]
Because at least one of $e_{\bar 1}(b)\otimes c$ or $f_{\bar 1}(b)\otimes c$
is nonzero, we have either 
 $e_{\bar 1}(b\otimes c) \neq 0$ or $f_{\bar 1}(b\otimes c)\neq 0$.
If $e_0(b\otimes c) = f_0(b\otimes c) = 0\neq f_0e_{\bar 1}(b\otimes c) $, then $e_0(c) =f_0(c) =0$
 so we must have 
  \[ f_0e_{\bar 1}(b\otimes c) = f_0(e_{\bar 1}(b) \otimes c) = f_0e_{\bar 1}(b) \otimes c
  \quand
  f_0e_{\bar 1}(b) \neq 0.
  \]
In this case, it follows from property (1) that 
 $  e_0(c\otimes d) = 0$ which can only hold if $e_0(d) =0$ since $e_0(c\otimes d) = c\otimes e_0(d)$.
 Likewise, 
  if $e_0(b\otimes c) = f_0(b\otimes c) = 0\neq e_0e_{\bar 1}(b) $, then  
  again $e_0(c) =f_0(c) =0$ so we must have 
  \[ e_0e_{\bar 1}(b\otimes c) = e_0(e_{\bar 1}(b) \otimes c) = e_0e_{\bar 1}(b) \otimes c
  \quand
  e_0e_{\bar 1}(b) \neq 0.
  \]
  In this case, it follows from property (2) that 
  $  f_0(c\otimes d) = 0$ which can only hold if $f_0(d) =0$ since $f_0(c\otimes d) = c\otimes f_0(d)$. 
  This lets us conclude that 
\[  e_{\bar 1}((b\otimes c)\otimes d) = e_{\bar 1}(b\otimes c)\otimes d.\]
Combining these equations gives the desired identity
\[
e_{\bar 1}(b \otimes (c\otimes d)) =
  e_{\bar 1}(b) \otimes c\otimes d =
  e_{\bar 1}(b \otimes c)\otimes d=
  e_{\bar 1}((b\otimes c)\otimes d)
.\]

\end{itemize} 
This shows that   $e_{\bar 1}$ 
  commutes with 
the bijection
$(\cB\otimes\cC) \otimes \cD \to \cB\otimes(\cC\otimes \cD)$,
which concludes our proof that this map is a $\qq_n$-crystal isomorphism.
\end{proof}

\subsection{Signature rules}

We have already encountered \defn{primed numbers} $1'<1<2'<2<\dots$
as formal symbols in our definition of shifted tableaux. From this point on, we 
 define $i' := i - \frac{1}{2} $ for $i \in \ZZ$ and set $\ZZ' := \ZZ -\frac{1}{2}$.
A \defn{primed word} is a finite sequence of primed numbers.
\defn{Removing the prime} from $i'$ corresponds to the   ceiling operation $\lceil\cdot\rceil$.

We identify a word $w=w_1w_2\cdots w_m$ having $w_i \in \{1'<1<\dots<n'<n\}$ 
with the tensor $w_1\otimes w_2 \otimes \cdots \otimes w_m \in (\BB^+_n)^{\otimes m}$.
This lets us evaluate $\weight(w)$, $e_i(w)$, and $f_i(w)$ for $i \in[n-1]$ using the definition of $(\BB^+_n)^{\otimes m}$. 
For example, the weight of $w$ is the vector whose $i$th component is the number of letters equal to $i$ or $i'$. 

Let $\unprime : (\BB^+_n)^{\otimes m} \to \BB_n^{\otimes m}$
be the map that replaces  $w=w_1w_2\cdots w_m$
by $\lceil w_1 \rceil \lceil w_2 \rceil \cdots \lceil w_m \rceil$.
Also set $\unprime(0)=0$.
The following is easy to check: 

\begin{lemma}\label{unprime-lem1}
%If  $i \in \{\bar 1,1,2,\dots,n-1\}$ then the map $\unprime$ commutes with the crystal operators $e_i$ and $f_i$ on $(\BB^+_n)^{\otimes m} $.
If $w \in (\BB^+_n)^{\otimes m}$ and $i \in \{\bar 1,1,2,\dots,n-1\}$ then 
\[ \unprime(e_i(w)) = e_i(\unprime(w)) \quand
\unprime(f_i(w)) = f_i(\unprime(w)).\]
\end{lemma}

It is possible to evaluate $e_i(w)$ and $f_i(w)$  directly from the formulas for  the tensor product $\otimes$, but this can be done more efficiently using the   following \defn{signature rule}.

Fix $i \in [n-1]$ and a primed word $w=w_1w_2\cdots w_m$. 
Mark each entry $w_j \in \{i',i\}$ by a right parenthesis ``)"  and each   $w_j \in \{i+1',i+1\}$ 
by a left parenthesis ``(". The \defn{$i$-unpaired indices} in $w$ are the indices $j \in [m]$ with 
$w_j \in \{i',i,i+1', i+1\}$ that are not the positions of matching parentheses.
In this case we refer to $w_j$ as an \defn{$i$-unpaired letter} of $w$.

 \begin{proposition}[\cite{GHPS,MT2023}]
 \label{sign-ex}
 Consider a primed word $w=w_1w_2\cdots w_m$. 
For each $i \in[n-1]$, one can compute  $e_i(w)$ and $f_i(w)$ using the following rules:
\begin{itemize}
\item[($e_i$)] If no $i$-unpaired index $j$ of $w$ has $w_j \in \{i+1', i+1\}$ then $e_i(w) := 0$.
\newline
Otherwise, if $j$ is the first such index, then $e_i(w) = w_1 \cdots (w_j-1)\cdots w_m$.

\item[($f_i$)] If no $i$-unpaired index $j$ of $w$ has $w_j \in \{i', i\}$ then $f_i(w) := 0$.
\newline
Otherwise, if $j$ is the last such index, then $f_i(w) = w_1 \cdots (w_j+1)\cdots w_m$.

\end{itemize}
The formulas for $e_0(w)$ and $f_0(w)$ are more straightforward: % (see \cite[Def.~4.2]{MT2023}):
\begin{itemize} 
\item[($e_0$)] If $w$ has  no $1'$ letters or if a $1$ appears before the first $1'$, then $e_{0}(w)=0$.
\newline
Otherwise, $e_{0}(w)$ is formed by changing the first $1'$ in $w$ to $1$.

\item[($f_0$)] If $w$ has  no $1$ letters or if a $1'$ appears before the first $1$, then $f_{0}(w)=0$.
\newline
Otherwise, $f_{0}(w)$ is formed by changing the first $1$ in $w$ to $1'$.

\end{itemize}
Finally, one can compute $e_{\bar 1}(w)$  and $f_{\bar 1}(w)$  as follows: %  (see \cite[Defs.~4.3 and 4.4]{MT2023}):
\begin{itemize}
\item[($e_{\bar 1}$)] Let $j,k\in[m]$ be minimal with $w_j \in \{2',2\}$ and $w_k \in \{1',1\}$.
\newline
 If $j$ does not exist or if $j>k$ then $e_{\bar 1}(w) =0$.
\newline
 If $j$ exists but $k$ does not, then $e_{\bar 1}(w) =  e_1(w)=w_1 \cdots (w_j-1)\cdots w_m$.
\newline
Otherwise, $e_{\bar 1} :  w=w_1 \cdots w_j \cdots w_k \cdots w_m \mapsto w_1 \cdots w_k \cdots (w_j-1)\cdots w_m.$
 \newline
This changes $w_jw_k = 2^\bullet 1^\circ$ to $1^\circ 1^\bullet$
where $\bullet $ and $\circ$ are arbitrary primes.
%$j$ and $k$ both exist with $j<k$,
%, then $e_{\bar 1}$ acts on $w$ as 
%\[ e_{\bar 1} :  w_1 \cdots w_j \cdots w_k \cdots w_m \mapsto w_1 \cdots w_k \cdots (w_j-1)\cdots w_m.\]
%if we write $w=w_1\cdots w_j \cdots w_k \cdots w_m$ then $e_{\bar 1}$ acts as
%\[
%e_{\bar 1} :  
%\begin{cases} 
%w_1 \cdots 2{\phantom{'}}  \cdots 1{\phantom{'}} \cdots w_m \mapsto w_1 \cdots 1{\phantom{'}} \cdots 1{\phantom{'}}\cdots w_m
%\\
%w_1 \cdots 2' \cdots 1' \cdots w_m \mapsto w_1 \cdots 1' \cdots 1'\cdots w_m
%\\
%w_1 \cdots 2{\phantom{'}} \cdots 1' \cdots w_m \mapsto w_1 \cdots 1' \cdots 1{\phantom{'}}\cdots w_m
%\\
%w_1 \cdots 2' \cdots 1{\phantom{'}} \cdots w_m \mapsto w_1 \cdots 1{\phantom{'}}\cdots 1'\cdots w_m.
%\end{cases}
%\]

\item[($f_{\bar 1}$)] Now let $j,k\in[m]$ be minimal with $w_j,w_k \in \{1',1\}$ and $j<k$.
\newline
 If $j$ does not exist or some $i \in [j-1]$ has $w_i \in \{2',2\}$ then $f_{\bar 1}(w) =0$.
\newline
If $j$ exists but $k$ does not, then $f_{\bar 1}(w) := f_1(w)= w_1 \cdots (w_j+1)\cdots w_m$.
\newline
Otherwise,  $f_{\bar 1} : w=w_1 \cdots w_j \cdots w_k \cdots w_m \mapsto w_1 \cdots (w_k+1) \cdots w_j\cdots w_m$
\newline
This changes $w_jw_k = 1^\circ 1^\bullet$ to $2^\bullet 1^\circ$
where $\bullet $ and $\circ$ are arbitrary primes.
%$j$ and $k$ both exist with $j<k$,
%, then $e_{\bar 1}$ acts on $w$ as 
%\[ e_{\bar 1} :  w_1 \cdots w_j \cdots w_k \cdots w_m \mapsto w_1 \cdots w_k \cdots (w_j-1)\cdots w_m.\]
%  $w=w_1\cdots w_j \cdots w_k \cdots w_m$
%  then $f_{\bar 1}$ acts as
%\[
%f_{\bar 1} :  
%\begin{cases} 
%   w_1 \cdots 1{\phantom{'}} \cdots 1{\phantom{'}}\cdots w_m
%   \mapsto 
%   w_1 \cdots 2{\phantom{'}}  \cdots 1{\phantom{'}} \cdots w_m
%\\
%   w_1 \cdots 1' \cdots 1'\cdots w_m \mapsto w_1 \cdots 2' \cdots 1' \cdots w_m 
%\\
%   w_1 \cdots 1' \cdots 1{\phantom{'}}\cdots w_m \mapsto w_1 \cdots 2{\phantom{'}} \cdots 1' \cdots w_m
%\\
%   w_1 \cdots 1{\phantom{'}}\cdots 1'\cdots w_m \mapsto w_1 \cdots 2' \cdots 1{\phantom{'}} \cdots w_m.
%\end{cases}
%\]

\end{itemize}
\end{proposition}

 \section{Results}
 
 This section contains our new results and is organized as follows.
 Section~\ref{dec-sect} introduces a $\qq_n$-crystal on \defn{primed decomposition tableaux}.
 Section~\ref{hl-sect} discusses the \defn{highest and lowest weight elements} for this crystal.
 Sections~\ref{ins-sect} and \ref{pro-sect} are concerned with the crystal-theoretic properties of an insertion algorithm for primed decomposition tableaux, extending a construction in \cite{GJKKK}. Sections~\ref{normal-sect} and \ref{comp-sect}
  derive several applications.

\subsection{Decomposition tableaux}\label{dec-sect}

A \defn{hook word} is a finite sequence of positive integers $w=w_1w_2\cdots w_n$ 
such that $w_1 \geq w_2 \geq \dots \geq w_m < w_{m+1} < w_{m+2} <\dots <w_n$ for some $m\in[n]$.
 Given such a hook word, let ${w\downarrow} := w_1w_2\cdots w_m$ denote the \defn{decreasing part} and 
 let ${w\uparrow} := w_{m+1}w_{m+2}\cdots w_n$ denote the \defn{increasing part}.

Fix a strict partition $\lambda$.
 A \defn{(semistandard) decomposition tableau} of shape $\lambda$ 
 is a shifted tableau $T$ of shape $\lambda$ such that if $\rho_i$ denotes row $i$ of $T$,
 then (1) each $\rho_i$ is a hook word and (2)  $\rho_i$ is a hook subword of maximal length in $\rho_{i+1}\rho_i$ for each $i \in [\ell(\lambda)-1]$.

This preceding definition follows \cite{GJKKK} but differs from \cite{CNO,Serrano}, where the opposite weak/strict inequality convention is used for hook words.
What we called a decomposition tableau is referred to as a \defn{reverse semistandard decomposition tableau} in \cite[Def.~2.8]{CNO}.
For a bijection between the two families of decomposition tableaux, see \cite[Thm.~3.9]{CNO}

Let $\DTab(\lambda)$ be the set of all decomposition tableaux of shape $\lambda$
and let $\DTab_n(\lambda)$ but the subset of such tableaux that have all entries in $[n]$.

 \begin{example}
Draw our tableaux in French notation, we have 
\[ %\ytab{ \none & 1 \\ 2 & 1 & 1 } \in \DTab_2((3,1))
\ytab{ \none & 1 & 1 \\ 2 & 2 & 1 } \in  \DTab((3,2))
\quad\text{but}\quad  \ytab{ \none & 1 & 1 \\ 2 & 2 & 3 } \notin  \DTab((3,2)).
\]
The second example is not a decomposition tableau because its row reading word 
$\rho_2\rho_1  = 11223$ contains the hook subword $1123$ which is longer than $\rho_1=223$.
 \end{example}
 
  The maximal hook subword condition in the definition of a decomposition tableau is satisfied if and only if 
 certain inequalities never hold for triples of entries in consecutive rows.
 In our experience, it is usually much easier to reason about decomposition tableaux conceived in terms of these inequalities.
 
 \begin{lemma}[{\cite[Prop.~2.3]{GJKKK}}] \label{decomptab-characterisation}
 Let $T$ be shifted tableau of shape $\lambda$ whose rows are each hook words.
 Then $T$ is a decomposition tableaux if and only if none of the following conditions holds for any $i \in [\ell(\lambda)-1]$ and $j,k \in [\lambda_{i+1}]$:
 \ben
% \item[(a)] $T_{ii} \leq T_{i+1,i+j}$; 
% \item[(b)] $T_{i+1,i+j} \geq T_{i+1,i+k} \geq T_{i,i+j}$ when $j<k$; or
% \item[(c)] $T_{i+1,i+k} < T_{i,i+j} < T_{i,i+k}$ when $j<k$.
 
  \item[(a)] $T_{i,i} \leq T_{i+1,i+k}$ or $ T_{i,i+j} \leq T_{i+1,i+k}  \leq  T_{i+1,i+j}$ when $j<k$, 
 \item[(b)] $T_{i+1,i+k} < T_{i,i} < T_{i,i+k}$ or $T_{i+1,i+k} < T_{i,i+j} < T_{i,i+k}$ when $j<k$.
  \een
 That is, we forbid rows $i$ and $i+1$ of $T$ from having configurations of entries 
 \[
 \ytabb{
 \none & \none & \cdots & b  \\
 \none & a  & \cdots & \  
 },
\quad
 \ytabb{
 \none & \cdots  & c & \cdots & b  \\
\ & \cdots & a & \cdots & \ 
 },
  \quad
 \ytabb{
 \none & \cdots  &x     \\
y & \cdots & z  
 },
 \quord
 \ytabb{
 \none & \cdots  &  & \cdots & x  \\
\ & \cdots & y & \cdots & z 
 }
 \]
  with $a\leq b \leq c$ and $x<y<z$.
  Here, the leftmost boxes are on the main diagonal and the ellipses ``$\cdots$'' indicate sequences of zero or more columns.
 \end{lemma}

Define the  \defn{middle element} of a hook word $w$ to be the last letter in the weakly decreasing subword $w\downarrow$.
Suppose $T$ is a decomposition tableau of strict partition shape $\lambda$.
We call any tableau  
formed by adding primes to the middle elements in a subset of rows in $T$
 a \defn{primed decomposition tableau} of shape $\lambda$.
Let $\DDTab(\lambda)$ denote the set of such tableaux $T$
and let $\DDTab_n(\lambda)$ be the subset consisting of those $T$
with all entries in $\{1'<1<\dots<n'<n\}$.

\begin{example}
$ \ytab{ \none & 1 \\ 2 & 1 & 2 },
$
$
\ytab{ \none & 1' \\ 2 & 1 & 2 },
$
$
\ytab{ \none & 1 \\ 2 & 1' & 2 },
$ 
$ 
\ytab{ \none & 1' \\ 2 & 1' & 2 }
$
are all in $\DDTab((3,1))$.
\end{example}
 
 It is useful to observe when  $\DTab_n(\lambda)$ and $\DDTab_n(\lambda)$ are nonempty.
 \begin{lemma}\label{nonempty-lem}
 Suppose $\lambda$ is a strict partition. Then the set $\DTab_n(\lambda)$ (equivalently, $\DDTab_n(\lambda)$)
 is nonempty if and only if $\lambda$ has at most $n$ nonzero parts.
 \end{lemma}
 
 \begin{proof}
 If $\ell(\lambda) \leq n$ then the shifted tableau of shape $\lambda$ with $n+1 - i$ in all boxes in row $i$
belongs to $\DTab_n(\lambda)\subseteq \DDTab_n(\lambda)$.
If $\ell(\lambda) > n$ then $\DTab_n(\lambda)$  and $\DDTab_n(\lambda)$ are empty 
as the diagonal entries of any $T \in \DTab_n(\lambda)$ must form a strictly decreasing sequence 
 of integers in $[n]$
  by Lemma~\ref{decomptab-characterisation}.
 \end{proof}
 
 The \defn{row reading word} of a shifted tableau $T$ is 
the word $\row(T)$ formed by reading the rows from left to right, but starting with last row.
The \defn{reverse row reading word} of  $T$ is  the reversal
of $\row(T)$; we denote this by $\revrow(T)$. 

\begin{example}
$\row\(\ytab{ \none & 1 \\ 2 & 1 & 1' }\) =1211'$
and
$\revrow\(\ytab{ \none & 2 & 1 \\ 2 & 2 & 3 }\) =32212$.
\end{example}

A \defn{crystal embedding} is a weight-preserving injective map $\phi : \cB \to \cC$ between crystals 
that commutes with all crystal operators, in the sense that $\phi(e_i(b)) = e_i(\phi(b))$ and $\phi(f_i(b)) = f_i(\phi(b))$
for all $b \in \cB$ when we set $\phi(0) =0$.
The following theorem extends \cite[Thm.~2.5(a)]{GJKKK}
from $\q_n$-crystals to $\qq_n$-crystals.

\begin{theorem}\label{revrow-thm}
Suppose $\lambda$ is a strict partition with at most $n$ parts.
There is a unique $\qq_n$-crystal structure on $\DDTab_n(\lambda)$ 
that makes $\revrow : \DDTab_n(\lambda)\to (\BB^+_n)^{\otimes |\lambda|}$
into a $\qq_n$-crystal embedding.
This structure restricts to a 
 $\q_n$-crystal  on $\DTab_n(\lambda)$
 for which $\revrow : \DTab_n(\lambda)\to \BB_n^{\otimes |\lambda|}$
is a $\q_n$-crystal embedding.
\end{theorem}

The relevant weight map has $x^{\weight(T)} = x^T$.
Figure~\ref{tab-fig} shows an example.

\begin{proof}
Let $T \in \DDTab_n(\lambda)$ and $i \in \{\bar 1,0,1,2,\dots,n-1\}$, and $w = \revrow(T)$.
When $e_i(w) =0$ define $e_i(T) =0$, and when $e_i(w) \neq 0$ define
$e_i(T)$ be the unique shifted tableau of shape $\lambda$ with $\revrow(e_i(T)) = e_i(w)$.
Define $f_i(T)$ analogously. We must verify that $e_i(T)$ and $f_i(T)$
are in 
$\DDTab_n(\lambda)\sqcup \{0\}$. 

This is almost self-evident when $i=0$. In this case
 $e_i(T)$ and $f_i(T)$ are either zero or formed by locating the boxes of $T$ containing $1'$ or $1$, and then toggling the prime on the first such box to appear in the reverse row reading word order.
The toggled box must  contain the middle element of its row, so the resulting tableau is  in $\DDTab_n(\lambda)$.

Now assume $i \in \{\bar 1,1,2,\dots,n-1\}$. 
Then $e_i$ and $f_i$ commute with $\unprime$ by Lemma~\ref{unprime-lem1},
so we know already that $\unprime(e_i(T)) = e_i(\unprime(w))$ and $\unprime(f_i(T))= f_i(\unprime(w))$
belong to $\DTab_n(\lambda) \sqcup \{0\}$.
We just need to explain why all primed entries in $e_i(T)$ and $f_i(T)$ are the middle elements of their rows. 
Suppose $i \in [n-1]$ and $e_i(T) \neq 0$. Then $e_i(T)$ is formed from $T$ by decrementing 
the entry in the box $(x,y)$ that contributes the first $i$-unpaired letter of $w$ equal to $i+1'$ or $i+1$. 
The defining properties of $(x,y)$ imply that box $(x,y-1)$ cannot contain $i'$ or $i$ in $T$
(as then $T_{xy}$ would be $i$-paired in $w$)
while box $(x,y+1)$ cannot contain $i+1'$ or $i+1$ in $T$ (as then $T_{x,y+1}$ would contribute an earlier $i$-unpaired letter to $w$).
Therefore, subtracting one from $T_{xy}$ does not change the locations of the middle elements of $T$, so the locations of the middle elements of $T$ and $e_i(T)$ are identical. As the locations of the primed entries of $T$ and $e_i(T)$ also coincide, we conclude that $e_i(T) \in \DDTab_n(\lambda)$.

If $i \in [n-1]$ and $f_i(T) \neq 0$,
then $f_i(T)$ is formed from $T$ by adding one to 
the entry in the box $(x,y)$ that contributes the last $i$-unpaired letter of $w$ equal to $i'$ or $i$. 
As in the previous case, one can deduce from these properties
that the locations of the middle elements are the same in $T$ and $f_i(T)$, as are the locations of the primed entries, so   $f_i(T) \in \DDTab_n(\lambda)$.

Next suppose $e_{\bar 1}(T) \neq 0$.
Define $j$ and $k$ as in part ($e_{\bar 1}$) of Proposition~\ref{sign-ex},
and let $(x_j,y_j)$ and $(x_k,y_k)$ be the boxes of $T$ that contribute entries $w_j$ and $w_k$ to $w$.
If $k$ does not exist 
then $(x_j,y_j)$ contains the middle element of its row in both $T$ and $e_{\bar 1}(T)$.
Then, as above, the locations of the middle elements of $e_{\bar 1}(T)$
are the same as in $T$, as are the locations of the primed elements, and this is enough to deduce that  $e_{\bar 1}(T) \in \DDTab_n(\lambda)$.

Assume $k$ exists. If $x_j < x_k$ then 
$(x_j,y_j)$ and $(x_k,y_k)$ contain the middle elements of their rows in both $T$ and $e_{\bar 1}(T)$. In this case,
the locations middle elements in the rows of these tableaux again coincide, and the set of primed boxes in $e_{\bar 1}(T)$ is the symmetric set difference of $\{(x_j,y_j),(x_k,y_k)\}$ and the set of primed boxes in $T$.
It follows that every primed box of $e_{\bar 1}(T)$ contains the middle element of its row, so $e_{\bar 1}(T) \in \DDTab_n(\lambda)$.

We can only have $x_j=x_k$ if $y_j = y_k +1$.
Then, the locations of the middle elements in $e_{\bar 1}(T)$ are derived from those of $T$
by removing $(x_k,y_k)$ and adding $(x_j,y_j)$.
In this case, outside row $x_j=x_k$,  the primed boxes 
in $e_{\bar 1}(T)$ are the
 same as in $T$ and only contain middle elements.
 In row $x_j=x_k$, if there is a primed box in $T$, then it occurs in column $y_k$ but moves to column $y_j=y_k+1$ in $e_i(T)$, so remains with the middle element.
 We conclude again that $e_{\bar 1}(T) \in \DDTab_n(\lambda)$.

 Finally suppose $f_{\bar 1}(T) \neq 0$.
Define $j$ and $k$ as in the part ($f_{\bar 1}$) of Proposition~\ref{sign-ex},
and let $(x_j,y_j)$ and $(x_k,y_k)$ be the boxes of $T$ that contribute entries $w_j$ and $w_k$ to $w$.
One checks that if
 $k$ does not exist 
then $(x_j,y_j)$ contains the middle element of its row in both $T$ and $f_{\bar 1}(T)$,
and the locations of the middle elements of $f_{\bar 1}(T)$
are the same as in $T$, as are the locations of the primed elements.
In turn, if $k$ exists and $x_j < x_k$ then
the locations of the middle elements $T$ and $f_{\bar 1}(T)$ coincide, while the set of primed boxes in $f_{\bar 1}(T)$ is the symmetric set difference of $\{(x_j,y_j),(x_k,y_k)\}$ and the set of primed boxes in $T$.

In the only remaining case, 
we must have $x_j=x_k$ and $y_j = y_k +1$.
Then,
outside row $x_j=x_k$,  the primed boxes 
in $f_{\bar 1}(T)$ are the
 same as in $T$ so are the positions of the middle elements in some subset of rows.
In row $x_j=x_k$, if there is a primed box in $T$, then it occurs in column $y_j$ but moves to column $y_k$ in $f_i(T)$, and so remains with the middle element.
From the observations, which exactly mirror the $e_{\bar 1}$ subcases,
we deduce that  $f_{\bar 1}(T) \in \DDTab_n(\lambda)$.
\end{proof}

\begin{figure}
\centerline{\input{qq3-dec-tableau-crystal.tex}}
\caption{Crystal graph of $\qq_3$-crystal $\DDTab_3(\lambda)$ for $\lambda=(2,1)$.
In this picture, 
solid blue and red arrows respectively indicate 
$ \xrightarrow{\ 1\ } $ and $ \xrightarrow{\ 2\ } $ edges while dotted green and dashed blue arrows 
indicate $ \xrightarrow{\ 0\ } $ and $ \xrightarrow{\ \bar 1\ }$ edges.
}
\label{tab-fig}
\end{figure}

Let the symbol $\unprime$  also denote the map $\DDTab_n(\lambda)\to\DTab_n(\lambda)$
 which removes the primes from the entries of a primed decomposition tableau.
The following is clear from Lemma~\ref{unprime-lem1} and Theorem~\ref{revrow-thm}.

\begin{lemma}\label{unprime-lem2}
If $T \in \DDTab_n(\lambda)$ and $i \in \{\bar 1,1,2,\dots,n-1\}$ then 
\[ \unprime(e_i(T)) = e_i(\unprime(T)) \quand
\unprime(f_i(T)) = f_i(\unprime(T)).\]
\end{lemma}

\subsection{Highest and lowest weights}\label{hl-sect}

An important property of many crystals is the existence of unique \defn{highest and lowest weight elements}.
For $\gl_n$-crystals, such elements are defined as 
the sources and sinks in the crystal graph:
  if $\cB$ is a $\gl_n$-crystal then $b\in\cB$ is \defn{highest weight} 
(respectively, \defn{lowest weight}) if  $e_i(b) = 0$
(respectively, $f_i(b) =0$)
 for all $i \in [n-1]$.
 The   definitions of highest weight elements for $\q_n$ and $\qq_n$-crystals from \cite{GJKKK,MT2023} are more technical, and given as follows.

Assume $\cB$ is a crystal  and $i$ is an index.
An \defn{$i$-string} in $\cB$ is a connected component in the subgraph of the crystal graph retaining only the $\xrightarrow{i}$ arrows. Let $\sigma_i : \cB \to \cB$ be the involution that reverses each $i$-string, so that  the first and last elements are swapped, the second and second-to-last elements are swapped, and so on. 
%We may express $\sigma_i$ by the formula
%\[\sigma_i(b) = e_i^{\varepsilon_i(b)}f_i^{\varphi_i(b)}(b) =  f_i^{\varphi_i(b)} e_i^{\varepsilon_i(b)}(b).\]
If $i \in [n-1]$ then swapping $\weight(b)_i$ and $\weight(b)_{i+1}$
gives $\weight(\sigma_i(b))$.

Assume $\cB$ is a $\q_n$-crystal.
Define $e_{\bar i}: \cB \to \cB \sqcup\{0\}$ and $  f_{\bar i} : \cB \to \cB \sqcup\{0\}$ for each index $2\leq i <n$ to be the ``twisted'' crystal operators
\be
\label{bar-i-eq} 
\ba
e_{\bar i} &:=  (\sigma_{i-1}   \sigma_i)    \cdots  (\sigma_2 \sigma_3) (\sigma_1 \sigma_2) e_{\bar 1} (\sigma_2 \sigma _1) (\sigma_3 \sigma_2) \cdots (\sigma_i   \sigma_{i-1}),
% = \sigma_{i-1}   \sigma_i   e_{\overline{i-1}}   \sigma_i   \sigma_{i-1},
\\
f_{\bar i} &:=  (\sigma_{i-1}   \sigma_i)    \cdots (\sigma_2 \sigma_3) (\sigma_1 \sigma_2) f_{\bar 1}( \sigma_2 \sigma _1) (\sigma_3 \sigma_2) \cdots (\sigma_i   \sigma_{i-1}),
% =  \sigma_{i-1}   \sigma_i   f_{\overline{i-1}}   \sigma_i   \sigma_{i-1} ,
\ea
\ee
using the convention that $\sigma_i(0) = 0$. Also define
 $\sigma_{w_0} : \cB \to \cB$ by
\be\label{w0-eq} \sigma_{w_0} := (\sigma_1) (\sigma_2\sigma_1)(\sigma_3\sigma_2\sigma_1) \cdots  (\sigma_{n-1} \cdots \sigma_2\sigma_1)\ee
and for each $i \in [n]$ set
 $ e_{\bar i'} := \sigma_{w_0}  f_{\overline{n-i}}  \sigma_{w_0}^{-1}
 $
 and
 $
 f_{\bar i'} := \sigma_{w_0}   e_{\overline{n-i}}   \sigma_{w_0}^{-1}.
 $

\begin{definition}%[{\cite[Def.~1.12]{GJKKK}}]
An element $b\in\cB$ is \defn{$\q_n$-highest weight}
if $e_i (b) =e_{\bar i}(b)= 0$ for all $i \in [n-1]$,
and  \defn{$\q_n$-lowest weight}
if $f_i (b) =f_{\bar i'}(b)= 0$ for all $i \in [n-1]$.
 \end{definition}
 
 Now assume that $\cB$ is $\qq_n$-crystal.
 For each $i \in[n]$ let 
 \be\ba
 e_0^{[i]} &:= \sigma_{i-1} \cdots \sigma_2 \sigma_1 e_0 \sigma_1 \sigma_2 \cdots \sigma_{i-1},
 \\
  f_0^{[i]} &:= \sigma_{i-1} \cdots \sigma_2 \sigma_1 f_0 \sigma_1 \sigma_2 \cdots \sigma_{i-1}.
  \ea
  \ee
 
 \begin{definition}%[{\cite[Def.~3.19]{MT2023}}]
An element $b$ in a $\qq_n$-crystal $\cB$ is \defn{$\qq_n$-highest weight}
if it is $\q_n$-highest weight with $e_0^{[i]}(b)= 0$ for all $i \in [n]$,
and \defn{$\qq_n$-lowest weight}
if it is $\q_n$-lowest weight with $f_0^{[i]}(b)= 0$ for all $i \in [n]$.
\end{definition}

\begin{remark}
For the special class of \defn{normal crystals} that will be discussed in Section~\ref{normal-sect},
 we  can characterize highest and lowest weight elements in a simpler way directly in terms of the relevant crystals graphs. See 
 Proposition~\ref{ranked-prop}.
\end{remark}

Let $\lambda$ be a strict partition with $\ell(\lambda)=k$.
The \defn{first border strip} of a shifted diagram $\SD_\lambda$
is the minimal subset $S$ containing $(1,\lambda_1)$ such that 
if $(i,j) \in S$ and $i\neq j$, then either
$(i+1,j)\in S$, or 
$(i,j-1) \in S$ when $(i+1,j)\notin \SD_\lambda$. 
Let $\SD_\lambda^{(1)}$ be the first border strip of $\SD_\lambda$.
The set difference $\SD_\lambda- \SD_\lambda^{(1)}$ is either empty when $k=1$
or equal to $\SD_\mu$ for a strict partition $\mu$ with $\ell(\mu) = k-1$.
For $i\in[k-1]$ let $\SD_\lambda^{(i+1)}$ be the first border strip of $\SD_\lambda - (\SD_\lambda^{(1)} \sqcup \cdots \sqcup \SD_\lambda^{(i)})$.
Finally, let $\Thighest_\lambda$ be the shifted tableau of shape $\lambda$ with all $i$ entries in $\SD_\lambda^{(i)}$. 

\begin{example}If $\lambda = (6,4,2,1)$ then
$\Thighest_{\lambda} = \ytab{
\none & \none & \none & 1 \\
\none & \none & 2 & 1 \\ 
\none & 3 & 2 & 1 & 1 \\
4 & 3 & 2 & 2 & 1 & 1
}
$.
\end{example}

Let $\Tlowest_\lambda$ be the shifted tableau of shape $\lambda$
whose $i$th row has all entries   $n+1-i$.
Form $\TTlowest_\lambda$ by adding a prime to the last entry in each row of $\Tlowest_\lambda$.

\begin{example} If $n=7$ and $\lambda = (6,4,2,1)$ then
\[\Tlowest_{\lambda} = \ytab{
\none & \none & \none & 4 \\
\none & \none & 5 & 5 \\ 
\none & 6 & 6 & 6 & 6 \\
7 & 7 & 7 & 7 & 7 & 7
}
\quand
\TTlowest_{\lambda} = \ytab{
\none & \none & \none & 4' \\
\none & \none & 5 & 5' \\ 
\none & 6 & 6 & 6 & 6' \\
7 & 7 & 7 & 7 & 7 & 7'
}
.\]
\end{example}

It is known that 
$\Thighest_\lambda$ and $\Tlowest_\lambda$
are the unique $\q_n$-highest weight and $\q_n$-lowest weight elements 
of $\DTab_n(\lambda)$
 \cite[Thm.~2.5(b)]{GJKKK}.
This property extends to the $\qq_n$-crystal $\DDTab_n(\lambda)$ as follows:

\begin{theorem}\label{highest-thm}
Suppose $\lambda$ is a strict partition with at most $n$ parts.
\ben
\item[(a)]   $\Thighest_\lambda$ is
the unique $\qq_n$-highest weight element of 
 $\DDTab_n(\lambda)$.
 
 \item[(b)]   $\TTlowest_\lambda$ is
the unique $\qq_n$-lowest weight element of 
 $\DDTab_n(\lambda)$.

 \een
 \end{theorem}

  \begin{proof}
Lemma~\ref{unprime-lem2} implies that
$\unprime$ sends $\qq_n$-highest and  lowest weight elements in $\DDTab_n(\lambda)$
to $\q_n$-highest and lowest weight elements in $\DTab_n(\lambda)$.
Therefore
$T \in \DDTab_n(\lambda)$
 is a $\qq_n$-highest weight element if and only if 
 \be\label{useful-eq1}
 \unprime(T) = \Thighest_\lambda\quand e_{0}^{[i]}(T)  = 0\text{ for all $i \in [n]$}.
 \ee
Likewise, $T \in \DDTab_n(\lambda)$ is a $\qq_n$-lowest  weight element
  if and only if 
 \be\label{useful-eq2}
 \unprime(T) =\Tlowest_\lambda\quand f_{0}^{[i]}(T)  = 0\text{ for all $i \in [n]$}.
 \ee

We first prove part (a). Fix $i \in [n-1]$, $j \in \{0, 1, 2, \dots, \ell(\lambda)-1\}$, 
and $T\in \DDTab_n(\lambda)$. 
Suppose the border strip $\SD_\lambda^{(i)}$ is exactly the set of boxes 
in $T$ containing $i$ or $i'$, and
the part of the border strip $\SD_\lambda^{(i+1)}$ after the first $j$ rows
is exactly the set of boxes 
in $T$ containing $i+1$ or $i+1'$.
Then in $\revrow(T)$, each box $(x,y)$ containing $i+1$ or $i+1'$ is paired with the box $(x+1, y+1)$, which contains $i$ or $i'$,
 while all boxes containing $i$ or $i'$ in the first $j+1$ rows are unpaired.
Therefore $\sigma_i$ acts on $T$ by changing every $i$ and $i'$ in the first $j+1$ rows of the border strip $\SD_\lambda^{(i)}$ to $i+1$ and $i+1'$ respectively.
For example, \[\sigma_1\(\ytab{
\none & \none & \none & 1' \\
\none & \none & 2 & 1 \\ 
\none & 3 & 2 & 1 & 1' \\
4 & 3 & 3 & 3 & 1 & 1
}\) 
= 
\ytab{
\none & \none & \none & 1' \\
\none & \none & 2 & 1 \\ 
\none & 3 & 2 & 2 & 2' \\
4 & 3 & 3 & 3 & 2 & 2
}.
\]
This example belongs to the case when $j=1$.

Now suppose $T$ has $\unprime(T) = \Thighest_\lambda$
and $T_{i, \lambda_i+i-1} = 1'$. By the observations in the previous paragraph with $j=0$,  $\sigma_{i-1}(T)$ is formed from $T$ by adding one to every $i-1$ or $i-1'$ entry in the first row.
By the same observations with $j=1$,
$\sigma_{i-2}\sigma_{i-1}(T)$ is formed from $\sigma_{i-1}(T)$ by adding one to every $i-2$ or $(i-2)'$ entry in the first two rows.
Continuing in this way for $j = 2, 3, \dots, i-2$, we see that box $(i,\lambda_i+i-1)$ in $\sigma_1\sigma_2\cdots\sigma_{i-1}(T)$
has entry $1'$ and this box contributes the first $1$ or $1'$ to the reverse row reading word. 
Therefore $e_0$ acts on $\sigma_1\sigma_2\cdots\sigma_{i-1}(T)$ by changing this $1'$ in to $1$,
so $e_0^{[i]}(T) \neq 0$ and $T$ is not $\qq_n$-highest weight.
See Figure~\ref{e0i-fig} for an example of the successive steps in this computation.

\begin{figure}
\[
\ba 
T &= \ytab{
\none & \none & \none & 1' \\
\none & \none & 2 & 1 \\ 
\none & 3 & 2 & 1 & 1 \\
4 & 3 & 2 & 2 & 1 & 1
}
\\
\sigma_3(T) &= \ytab{
\none & \none & \none & 1' \\
\none & \none & 2 & 1 \\ 
\none & 3 & 2 & 1 & 1 \\
4 & 4 & 2 & 2 & 1 & 1
}
\\
\quad \sigma_2\sigma_3(T) &= \ytab{
\none & \none & \none & 1' \\
\none & \none & 2 & 1 \\ 
\none & 3 & 3 & 1 & 1 \\
4 & 4 & 3 & 3 & 1 & 1
} \\ \sigma_1\sigma_2\sigma_3(T) &= \ytab{
\none & \none & \none & 1' \\
\none & \none & 2 & 2 \\ 
\none & 3 & 3 & 2 & 2 \\
4 & 4 & 3 & 3 & 2 & 2
}
\\
e_0\sigma_1\sigma_2\sigma_3(T) &= \ytab{
\none & \none & \none & 1 \\
\none & \none & 2 & 2 \\ 
\none & 3 & 3 & 2 & 2 \\
4 & 4 & 3 & 3 & 2 & 2
} \neq 0.
\ea
\]
\caption{Intermediate steps showing $e_0^{[i]}(T)\neq 0$ for $i = 4$}\label{e0i-fig}
\end{figure}

The only remaining candidate for a $\qq_n$-highest weight element is $\Thighest_\lambda$. 
We do have   $e_0^{[i]}(\Thighest_\lambda) = 0$ for all $i \in [n]$, since $e_0^{[i]}$ removes a prime from one box if it does not act as zero,
and $\Thighest_\lambda$ already has no primes.

Now we prove part (b). Fix $i \in [n-1]$ and $T\in \DDTab_n(\lambda)$.
Suppose that the $(n-i+1)$-th row of $T$ is exactly the set of all boxes in $T$ containing $i$ or $i'$.
Assume additionally that all entries in the $(n-i)$-th row of $T$ are $i+1$ or $i+1'$
and the only other boxes of $T$ with these entries are in the first $n-i$ rows of the first border strip $\SD_\lambda^{(1)}$.

Then, in $\revrow(T)$, each box containing $i+1$ in the first $\lambda_{n-i+1}$ boxes in the $(n-i)$-th row is paired with a box in the $(n-i+1)$-th row, which contains $i$ or $i'$. The remaining boxes containing $i+1$ or $i+1'$ in the first $(n-i)$ rows are left unpaired.
Therefore $\sigma_i$ acts on $T$ by changing every $i+1$ or $i+1'$ in the last $\lambda_{n-i} - \lambda_{n-i+1}$ boxes of the $(n-i)$-th row and in the boxes of the first $(n-i)$ rows of $\SD_\lambda^{(1)}$ to $i$ or $i'$, respectively. 
For example,  we have 
\[ \sigma_6\(\ytab{
\none & \none & \none & 4 \\
\none & \none & 5 & 5' \\ 
\none & 6 & 6 & 6 & 6 \\
7 & 7 & 7 & 7 & 7 & 7'
}\)
= 
\ytab{
\none & \none & \none & 4 \\
\none & \none & 5 & 5' \\ 
\none & 6 & 6 & 6 & 6 \\
7 & 7 & 7 & 7 & 6 & 6'
}
\]
which corresponds to the case when $n=7$ and $i=6$ so $n+1-i = 2$,
as well as
\[
\sigma_4\(\ytab{
\none & \none & \none & 4 \\
\none & \none & 5 & 5' \\ 
\none & 6 & 6 & 5 & 5 \\
7 & 7 & 7 & 7 & 7 & 7'
}\)
=
\ytab{
\none & \none & \none & 4 \\
\none & \none & 5 & 4' \\ 
\none & 6 & 6 & 4 & 4 \\
7 & 7 & 7 & 7 &7 & 7'
}
\]
which corresponds to the case when $n = 7$ and $i=4$ so $n+1-i = 4$.

Now suppose $T$ has $\unprime(T) = \Tlowest_\lambda$
and $T_{n-i+1, \lambda_{n-i+1}+n-i} = i$. By the observations in the previous paragraph, $\sigma_{i-1}(T)$ is formed from $T$ by subtracting one from every $i$ or $i'$ entry located in the last $\lambda_{n-i+1} - \lambda_{n-i+2}$ boxes of the $(n-i+1)$th row, which are also located in $\SD_\lambda^{(1)}$.
By the same observations.
$\sigma_{i-2}\sigma_{i-1}(T)$ is formed from $\sigma_{i-1}(T)$ by subtracting one from every $i-1$ or $(i-1)'$ entry in the last $\lambda_{n-i+1} - \lambda_{n-i+2}$ boxes of the $(n-i+1)$th row, and the last $\lambda_{n-i+2} - \lambda_{n-i+3}$ boxes of the $(n-i+2)$th row.
Continuing this process until $\sigma_{n-\ell(\lambda)+1}$, we observe that $U = \sigma_{n-\ell(\lambda)+1}\sigma_{n-\ell(\lambda)+2} \dots \sigma_{i-2} \sigma_{i-1} (T)$ is formed from $T$ by changing the part of the border strip $\SD_\lambda^{(1)}$ after the first $n-i$ rows to $n-\ell(\lambda)+1$ or $n-\ell(\lambda)+1'$, while keepings the locations of the primes the same as in $T$. In particular, the box $(n-i+1, \lambda_{n-i+1}+n-i)$ now contains the entry $n-\ell(\lambda)+1$.

Notice that $U$ does not have any entry equal to $n-\ell(\lambda)$ or $(n-\ell(\lambda))'$, and the boxes in $\SD_\lambda^{(1)}$ after the first $n-i$ rows consist of all the boxes of $U$ with entries $n-\ell(\lambda)+1$ or $n-\ell(\lambda)+1'$, and hence these entries are all $(n-\ell(\lambda))$-unpaired. Therefore by successively applying $\sigma_{n-\ell(\lambda)}$, $\sigma_{n-\ell(\lambda)-1}$, $\dots$, $\sigma_1$, we can conclude that $\sigma_1\sigma_2 \dots \sigma_{i-1}(T)$ is formed from $T$ by changing the part of the border strip $\SD_\lambda^{(1)}$ after the first $n-i$ rows to $1$ or $1'$, while maintaining the locations of the prime boxes. 
In particular, the box $(n-i+1, \lambda_{n-i+1}+n-i)$ in $\sigma_1\sigma_2 \dots \sigma_{i-1}(T)$ contains the entry $1$, and this box contributes the first $1$ or $1'$ to the reverse row reading word. 
Therefore applying $f_0$ changes this box to $1'$, and hence 
$ f_0^{[i]}(T) \neq 0$, so $T$ is not $\qq_n$-lowest weight.
See Figure~\ref{f0i-fig} for an example of the successive steps in this computation.

\begin{figure}
\[
\ba
T &=  \ytab{
\none & \none & \none & 4' \\
\none & \none & 5 & 5' \\ 
\none & 6 & 6 & 6 & 6' \\
7 & 7 & 7 & 7 & 7 & 7
}
\\
\sigma_6(T) &= \ytab{
\none & \none & \none & 4' \\
\none & \none & 5 & 5' \\ 
\none & 6 & 6 & 6 & 6' \\
7 & 7 & 7 & 7 & 6 & 6
} 
\\
\sigma_5\sigma_6(T) &= \ytab{
\none & \none & \none & 4' \\
\none & \none & 5 & 5' \\ 
\none & 6 & 6 & 5 & 5' \\
7 & 7 & 7 & 7 & 5 & 5
} 
\\ 
\sigma_4\sigma_5\sigma_6(T) &= \ytab{
\none & \none & \none & 4' \\
\none & \none & 5 & 4' \\ 
\none & 6 & 6 & 4 & 4' \\
7 & 7 & 7 & 7 & 4 & 4
} 
\ea
\qquad
\ba
\sigma_3\sigma_4\sigma_5\sigma_6(T) &= \ytab{
\none & \none & \none & 3' \\
\none & \none & 5 & 3' \\ 
\none & 6 & 6 & 3 & 3' \\
7 & 7 & 7 & 7 & 3 & 3
} 
\\
 \sigma_2 \sigma_3\sigma_4\sigma_5\sigma_6(T) &= \ytab{
\none & \none & \none & 2' \\
\none & \none & 5 & 2' \\ 
\none & 6 & 6 & 2 & 2' \\
7 & 7 & 7 & 7 & 2 & 2
}
\\
\sigma_1\sigma_2 \sigma_3\sigma_4\sigma_5\sigma_6(T) &= \ytab{
\none & \none & \none & 1' \\
\none & \none & 5 & 1' \\ 
\none & 6 & 6 & 1 & 1' \\
7 & 7 & 7 & 7 & 1 & 1
}
\\
f_0\sigma_1\sigma_2 \sigma_3\sigma_4\sigma_5\sigma_6(T) &= \ytab{
\none & \none & \none & 1' \\
\none & \none & 5 & 1' \\ 
\none & 6 & 6 & 1 & 1' \\
7 & 7 & 7 & 7 & 1 & 1'
}\neq 0.
\ea
\]
\caption{Intermediate steps showing $f_0^{[i]}(T)\neq 0$ for $i = 7$}\label{f0i-fig}
\end{figure}

The only remaining candidate for a $\qq_n$-lowest weight element is $\TTlowest_\lambda$. 
We do have   $f_0^{[i]}(\TTlowest_\lambda) = 0$ for all $i \in [n]$, since $f_0^{[i]}$ adds a 
prime to one box if it does not act as zero,
and $\TTlowest_\lambda$ already has the maximum number of primed entries for an element of $\DDTab_n(\lambda)$.
 \end{proof}

\subsection{Decomposition insertion}\label{ins-sect}

This section introduces a ``primed'' extension of Grantcharov et al.'s  insertion scheme from \cite[\S3]{GJKKK}.
This algorithm embeds each tensor power of the standard $\qq_n$-crystal in a disjoint union of decomposition tableau crystals.

\begin{definition}\label{decomp-insert-algo}
Suppose $T$ is a primed decomposition tableau and $x \in \ZZ\sqcup\ZZ'$.
Let   $x \toQQ T$ be the tableau formed
by the following insertion procedure. 
\bei
\item On step $i$ of the algorithm, a number $a_i$ is inserted into row $i$ of $T$, starting with $a_1:=x$ inserted into the first row.

\item To compute the insertion 
 on step $i$, set $a = \lceil a_i \rceil$ and remove any prime from the middle element $m_i$ of row $i$ (if the row is nonempty). The (unprimed) number $a$ is added to the end of the (now unprimed) row if this creates a hook word; otherwise, $a$ replaces the leftmost entry $b$ from the increasing part of the row with $b\geq a$,  and then $b$ replaces the leftmost entry $c$ from the weakly decreasing part of the row with $c < b$.
 
\item Now we must decide the value of $a_{i+1}$ and whether to add a prime to the middle element of row $i$.
There are two cases:
\ben
\item[(1)] Suppose row $i$ was initially empty, or the location of the middle element has moved (necessarily to the right). If $a_i \in \ZZ'$ then add a prime to the new middle element.
If a box was added to the end of row, then the algorithm
halts at this step and we say the insertion is \defn{even} if 
row $i$ was initially empty or $m_i \in \ZZ$, and \defn{odd} if 
row $i$ was not initially empty and $m_i \in \ZZ'$.
Otherwise,
we set $a_{i+1} =c$ when $m_i \in \ZZ$ and   $a_{i+1} =c'$ when $m_i \in \ZZ'$. For example:
 \[
 \ytab{ 4 & 2 & 2& 1^{\circ} & 3} \leftarrow 1^\bullet =a_i \quad \leadsto\quad 
 a_{i+1} = 2^\circ \leftarrow \ytab{ 4 & 3 & 2& 1 & 1^\bullet}.
 \]
 Here $\circ$ and $\bullet$ indicate arbitrary, unspecified choice of primes.

\item[(2)] Suppose   the location of the middle element in row $i$ has not changed.
If $m_i \in \ZZ'$ then  add a prime to the middle element.
If a box was added to the end of the row, then the algorithm
halts at this step and we say the insertion is \defn{even} if $a_i \in \ZZ$ and \defn{odd} if $a_i \in \ZZ'$.
Otherwise,
 set $a_{i+1} =c$ when $a_i \in \ZZ$ and   $a_{i+1} =c'$ when $a_i \in \ZZ'$. For example:
 \[
 \ytab{ 4 & 2 & 2& 1^{\circ} & 3} \leftarrow 3^\bullet =a_i \quad \leadsto\quad 
 a_{i+1} = 2^\bullet \leftarrow \ytab{ 4 & 3 & 2& 1^\circ & 3}.
 \]
\een
\eei
\end{definition}

\begin{remark}\label{simpler-remk}
If $T$ is an unprimed decomposition tableau and $x\in \ZZ$, then 
$x \toQ T$ is obtained from the following simpler procedure, which is described in both \cite[Def.~3.4]{GJKKK} and \cite[Def.~2.9]{CNO}.
On each step, a number $a$ is inserted into a row of $T$, starting with $x$ into the first row.
If adding $a$ to the end of the row yields a hook word, then we do this and halt.
Otherwise, $a$ replaces the   leftmost element $b$ of the increasing part  of the row with $b \geq a$,
then $b$ replaces the leftmost element $c$ of the weakly decreasing part of the row with $c < b$, and then we insert $c$ into the next row.
\end{remark}
 
\begin{definition}
Given any primed word $w=w_m\cdots w_2w_1$, form 
\[\Pqq(w) :=   w_m \toQQ (\cdots\toQQ  (w_2 \toQQ (w_1 \toQQ \emptyset))\cdots )\]
by inserting the letters of $w$ read right to left into the empty tableau $\emptyset$.
Let $\Qqq(w)$ be the tableau with the same shape as $\Pqq(w)$
that has $i$ (respectively, $i'$) in the box added by $w_i \toQQ$ if this insertion is even (respectively, odd).
\end{definition}

\begin{example}
For $w = 
4'4332'3'32'1'$
we have
\[
\Pqq(w) = \ytab{
\none & \none & 1'  \\
\none & 2 & 2' & 3 \\
4 & 3 & 3 & 3 & 4 
}
\quand
\Qqq(w) = \ytab{
\none & \none & 7 \\
\none & 4 & 5' & 9' \\
1 & 2' & 3 &6 & 8 
}.
\]
\end{example}

The following is easy to check by induction on word length:

\begin{lemma}\label{unprime-lem3} If $w$ is a primed word 
then $\unprime(\Pqq(w)) = \Pqq(\unprime(w))$
and $\unprime(\Qqq(w)) = \Qqq(\unprime(w))$.
\end{lemma}

A shifted tableau with $n$ boxes is \defn{standard}
if its rows and columns are strictly increasing and it has exactly one entry equal to $i'$ or $i$
for each $i \in [n]$.

\begin{remark}\label{r-rmk}
Let $w^\r$ be the reverse of a word $w$.
On unprimed words, the map $w \mapsto (\Pqq(w^\r),\Qqq(w^\r))$
coincides with \cite[Def.~4.1]{GJKKK} and gives a bijection
 to pairs 
 $(P,Q)$ where  $P$ is an (unprimed) decomposition tableau and $Q$ is a standard shifted tableau of the same shape with no  primed entries.
 This map is called \defn{reverse semistandard Kra\'skiewicz insertion} in \cite{CNO}.
\end{remark}

\begin{proposition}\label{qq-bij-thm}
The map $w\mapsto (\Pqq(w),\Qqq(w))$ is a bijection from the set of 
all words with letters in $\{1'<1<2'<2<\dots\}$
 to the set of pairs $(P,Q)$ of shifted tableaux with the same shape
such that $P$ is a primed decomposition tableau and $Q$ is a standard shifted tableau with no  primed diagonal entries.
\end{proposition}

\begin{proof}
Fix an unprimed word $v$ and let $T = \Pqq(v)$ and $U = \Qqq(v)$.
Let $\cW$ be the set of primed words $w$ with $\unprime(w) = v$.
Then let $\cT$ be the set of pairs $(P,Q)$ of shifted tableaux with the same shape
such that $P$ is a primed decomposition tableau $Q$ is a standard shifted tableau with no  primed diagonal entries, and with $\unprime(P) = T$ and $\unprime(Q) = U$.
 
By Lemma~\ref{unprime-lem3} and Remark~\ref{r-rmk},
the operation $w \mapsto (\Pqq(w),\Qqq(w))$ is a map $\cW \to \cT$,
and it suffices to show that this map is a bijection.
For this, we regard $\cW$ and $\cT$ as $\FF_2$-vector spaces
in which  the zero elements are $v$ and $(T,U)$, and in which addition in computed by ``summing'' the primes of corresponding letters or boxes entry-wise,
where a prime plus a prime yields no prime, as does no prime plus no prime,
while a prime plus no prime yield a prime.

Now observe that $\cW$ and $\cT$ have the same finite dimension, and the map $w \mapsto (\Pqq(w),\Qqq(w))$ is $\FF_2$-linear with trivial kernel, so it is a bijection.
\end{proof}

A map $\phi : \cB \to \cC$ between ($\gl_n$, $\q_n$, or $\qq_n$) crystals is a \defn{quasi-isomorphism}
if for each full subcrystal $\cB' \subseteq \cB$ there is a full subcrystal $\cC' \subseteq \cC$
such that $\phi|_{\cB'}$ is an isomorphism $\cB' \to \cC'$.
The $\q_n$ part of the following  result is \cite[Thm.~4.5]{GJKKK}.

\begin{theorem}\label{main-thm}
The map $\Pqq$ defines $\q_n$ and $\qq_n$ crystal quasi-isomorphisms
\[ 
\BB_n^{\otimes m} \to \bigsqcup_{\substack{\text{strict }\lambda \vdash m \\ \ell(\lambda) \leq n}} \DTab_n(\lambda)
\quand
(\BB_n^+)^{\otimes m} \to \bigsqcup_{\substack{\text{strict }\lambda \vdash m \\ \ell(\lambda) \leq n}} \DDTab_n(\lambda).\]
Moreover, 
the full $\q_n$-subcrystals of $\BB_n^{\otimes m} $ and the full $\qq_n$-subcrystals of $(\BB_n^+)^{\otimes m} $ are the subsets on which $\Qqq$ is constant.
\end{theorem}

The next section is devoted to the proof of this theorem.

\subsection{Proof of Theorem~\ref{main-thm}}\label{pro-sect}

We adopt the convention that $ \Pqq(0) =\unprime(0) = e_i(0) = f_i(0) = 0$.

\begin{lemma}\label{main-lem1}
Suppose $w$ is a primed word and $i \in \{\bar 1,1,2,\dots,n-1\}$. Then
\[\unprime(\Pqq(e_i(w))) = \unprime(e_i(\Pqq(w))).\]
\end{lemma}

\begin{proof}
On unprimed words, our definition of $\Pqq$ coincides with the insertion tableau
in \cite[Def.~4.1]{GJKKK} (but with the order of insertion reversed), and it is shown in the proof of \cite[Thm.~4.5]{GJKKK} then this insertion tableau commutes with the $\q_n$-crystal operators on $\BB_n^{\otimes m}$. As our $\q_n$-tensor product follows the anti-Kashiwara convention, which is the reverse of the convention in \cite{GJKKK}, this means that
 \be\Pqq(e_i(v)) = e_i(\Pqq(v)) \text{ when }v \in \BB_n^{\otimes m}.
 \ee
Combining  this with
 Lemmas~\ref{unprime-lem1}, \ref{unprime-lem2}, and \ref{unprime-lem3}
 gives 
 \[\ba
 \unprime(\Pqq(e_i(w))) &=
 \Pqq(\unprime(e_i(w))) \\&=
  \Pqq(e_i(\unprime(w))) \\&=
    e_i(\Pqq(\unprime(w))) \\&=
                e_i(\unprime(\Pqq(w))) =
                \unprime(e_i(\Pqq(w))).\ea
                \]
\end{proof}

Recall that the string length $\varepsilon_i$ (respectively, $\varphi_i$) counts the 
number of $i$-unpaired entries equal to $i+1'$ or $i+1$
(respectively, $i'$ or $i$) in a primed word or in the reverse row reading word of a
primed decomposition tableau.

\begin{lemma}\label{main-lem2}
Fix  $w=w_1w_2\cdots w_m \in (\BB_n^+)^{\otimes m} $ with $m>0$
and  $i \in [n-1]$ with $e_i(w) \neq 0$. 
Let $v =w_2w_3\cdots w_m$ and $P = \Pqq(v)$,
and assume that
\ben
\item[(a)] one has $\varepsilon_i(v)  =\varepsilon_i(P)$ 
and $\varphi_i(v)  =\varphi_i(P)$, and
\item[(b)] if $e_i(v)\neq 0$ then
$ \Pqq(e_i(v)) = e_i(P)$ and $ \Qqq(e_i(v)) = \Qqq(v).$
\een
Then 
$\Pqq(e_i(w)) =e_i(\Pqq(w))$
and 
$\Qqq(e_i(w)) = \Qqq(w)$.
\end{lemma}

\begin{proof} 
To show that $\Pqq(e_i(w)) =e_i(\Pqq(w))$,
it suffices by Lemma~\ref{main-lem1} to show that the two tableaux have primed entries in 
exactly the same locations.
As the primed boxes in $e_i(\Pqq(w))$ are in the same locations as in $\Pqq(w)$, though
with different entries,
it is enough to show that $\Pqq(e_i(w)) $ and $\Pqq(w)$ have primed entries in 
exactly the same locations.
This is what we will actually check, along with $\Qqq(e_i(w)) = \Qqq(w)$.

By the definition of the $\gl_n$-crystal tensor product,
we have either $e_i(w) = w_1 e_i(v)$ or $e_i(w) = e_i(w_1)v$.
This divides the proof into two main cases.

First assume $e_i(w) =e_i(w_1) v$.
Then
  $\varepsilon_i(w_1) > \varphi_i(v)$,
which can only happen if $\varepsilon_i(w_1) = 1$ and $\varphi_i(v)=0$,
 so 
 $w_1 \in \{i+1',i+1\}$ and there are no unpaired
entries equal to $i+1$ or $i+1'$ in 
$v$, or in $\revrow(P)$ since $\varphi_i(P) = \varphi_i(v)$
by hypothesis.
Thus
$e_i(w) = (w_1-1)v$ and $
 \Pqq(e_i(w)) = (w_1-1) \toQQ P.
$
We now compare the effect of inserting $w_1$
versus $w_1-1$ into $P$:
\bei
\item If adding $w_1$ to the end of the first row and removing all primes 
yields a hook word, then the same is true of $w_1-1$ 
since the row cannot end in $i'$ or $i$
as $\varphi_i(\revrow(P)) = 0$.

\item If $w_1$ and $w_1-1$ bump the same
entry from the increasing part of the first row, then all subsequent
steps of the insertion processes are identical.

\item The only way that $w_1$ and $w_1-1$ can bump different entries
from the increasing part of the first row is if this part contains
the number $i$. But then, as $\varphi_i(\revrow(P)) = 0$, the increasing part must also contain $i+1$,
while the weakly decreasing part cannot contain any entries 
equal to $i'$ or $i$.
This means that even if 
$w_1$ and $w_1-1$ bump different entries
from the increasing part of the first row, the
same entries will be bumped from the weakly decreasing part of the row,
all subsequent
steps of the insertion process will be identical,
and the middle position in the first row will either change in both cases
or remain the same in both cases.
\eei
In each of these situations, since $w_1 $ and $w_1-1$ are either both primed or both unprimed, we have
 $\Qqq(e_i(w)) = \Qqq(w)$ and 
 the locations of the primed boxes in $\Pqq(e_i(w))$ and $\Pqq(w)$
are the same.
Therefore if
$e_i(w) =  e_i(w_1) v$
 then the desired claims hold.

Now suppose  $e_i(w) = w_1 e_i(v)$
so that $\varepsilon_i(w_1) \leq \varphi_i(v)$ and $e_i(v) \neq 0$.
This can only happen if either 
$w_1 \notin \{i+1',i+1\}$ or if there is some $i$-unpaired 
entry equal to $i'$ or $i$ in $v$ (and hence also in $\revrow(P)$, since $\varphi_i(v)  =\varphi_i(P)$).
Then  \[ \Pqq(e_i(w)) = w_1 \toQQ \Pqq(e_i(v))
= w_1 \toQQ e_i(P)
\]
by hypothesis, and  $e_i$ acts on $P$ by subtracting one from 
the  first $i$-unpaired entry equal to $i+1'$ or $i+1$ in the reverse row reading word order.
Suppose this entry is in box $(j,k)$. 
We now compare the two insertion processes to construct $w_1 \toQQ P=\Pqq(w) $ and $w_1 \toQQ e_i(P)=\Pqq(e_i(w)) $.

If the process to insert $w_1 \toQQ P$ halts
before row $j$, then  the insertion process for $w_1 \toQQ e_i(P)$
proceeds by exactly the same steps (as the first $j-1$ rows
of $P$ and $e_i(P)$ are the same)
 and also halts before row $j$.
As we have $\Qqq(e_i(v)) = \Qqq(v)$ by hypothesis,
it follows in this case that 
 $\Qqq(e_i(w)) = \Qqq(w)$,
 and that
 the primed boxes in $\Pqq(e_i(w))= w_1 \toQQ e_i(P)$ and $\Pqq(w)=w_1 \toQQ P$ have the same locations.
 
 Assume the insertion process for
 $w_1 \toQQ P$ reaches row $j$, so that the same is true for the process
 inserting $w_1 \toQQ e_i(P)$.
 Let $a$ be the number inserted into row $j$ in both cases.
Then $a$ is either $w_1$ when $j=1$, or 
an entry from the weakly decreasing part of the previous row, possibly
with its prime toggled.

Since $P_{jk}$ is the first $i$-unpaired entry equal
 to $i+1'$ or $i+1$ in $\revrow(P)$,
 we can only have 
  $a \in \{i+1',i+1\}$ 
  if there is some entry $P_{jl} \in \{i',i\}$ for it to pair with 
  in the reverse row reading word order, where $k<l$. 
 Thus, if   $a \in \{i+1',i+1\}$  
 then the boxes to the right of $(j,k)$ in $P$
must contain at least one entry in $\{i',i\}$
 and no entries in $\{i+1',i+1\}$,
 or three not necessarily consecutive entries
 going left to right of the form $i,i,i+1$ or $i,i',i+1$.
 Note that the these boxes contain the same entries in $P$ as in $e_i(P)$.

Suppose adding $a$ to the end of row $j$ of $P$ and removing all primes
 yields a hook word.
  Then the same is true of $e_i(P)$, and 
 this addition
  changes the location of the middle element in $P$
 if and only if it does so in $e_i(P)$, 
since it cannot happen that $(j,k)$ is the last element of the row
 when $a \in \{i+1',i+1\}$.
 It follows that $\Qqq(e_i(w))$ and $\Qqq(w)$ are
 both formed by adding the same new entry to $\Qqq(e_i(v)) = \Qqq(v)$,
 so $\Qqq(e_i(w)) = \Qqq(w)$.
Moreover,
 the primed boxes in $\Pqq(e_i(w)) $ and $\Pqq(w) $ have the same locations.
 
 Suppose 
 we are not in the previous case 
 and $(j,k)$ is in the increasing part of row $j$ of $P$.
 Then $(j,k)$ is also in the increasing part of row $j$ of $e_i(P)$,
 as otherwise we would have $P_{j,k-1} \in \{i',i\}$ and this number would be $i$-paired with $P_{jk}$ in $\revrow(P)$.
All entries to the left of $(j,k)$ in row $j$ of $P$ must be greater than $i+1$,
and the entries to the right cannot include $i'$ or $i$ as 
then $P_{jk}$ would be $i$-paired in $\revrow(P)$.
 This means that row $j$ of $P$ must not contain 
 any entries equal to $i'$ or $i$. 
 
 As a consequence, we must have
  $a \notin\{i+1',i+1\}$,
 so if $a$ does not bump box $(j,k)$ in $P$ then it also does not bump box 
 $(j,k)$ in $e_i(P)$; in this event,   the two insertion
 processes bump the same entries in every subsequent row, so it is clear that 
  $\Qqq(e_i(w)) = \Qqq(w)$
 and that
 the primed boxes in $\Pqq(e_i(w)) $ and $\Pqq(w)$ have the same locations.
We reach the same conclusion if $a$ does bump box $(j,k)$ in $P$, as then
 $a$ also must bump box $(j,k)$ in $e_i(P)$, 
 and in each case the same entry will then be bumped from the 
 weakly decreasing part of the row, and the middle position will change
 in one insertion if and only if it changes in the other.
 
There is now just one case left to consider. Namely,
assume that $a$ bumps some entry $b$ in the increasing part of 
 row $j$, and that $(j,k)$ is in the row's weakly decreasing part
 in   $P$ (and also in $e_i(P)$).
 Since we cannot have $a \in \{i+1',i+1\}$ if $(j,k)$ is  the middle position of row $j$
 of $P$,
 the middle position of row $j$ changes when inserting $w_1 \to P$ if and only if 
 it changes when inserting $w_1 \to e_i(P)$.
  We are left with three subcases:
\bei
\item If $b\leq i$ or if $b$ is greater than some entry 
to the left of $(j,k)$ in row $j$ of $P$,
 then $b$ bumps the same entry from the weakly decreasing 
 part of row $j$ in both $P$  and $e_i(P)$;
 then, as above, the   two insertion processes will
 bump the same entries in every subsequent row, so  
  $\Qqq(e_i(w)) = \Qqq(w)$
 and  
 the primed boxes in $\Pqq(e_i(w)) $ and $\Pqq(w)$ have the same locations.

\item Suppose $b=i+1$. The number $b$ must originate in some box $(j,l)$ of $P$
with $k<l$. 
Since $P_{jk}$ is the first $i$-unpaired entry equal to $i+1'$ or $i+1$
in $\revrow(P)$, some box between columns $k$ and $l$
of row $j$ in $P$ (and also in $e_i(P)$) must contain $i'$ or $i$.
If this box is in the increasing part of the row,
then $b$ could only be bumped by $a \in \{i+1',i+1\}$,
which would require a second box to the right of $(j,k)$ to contain $i$.
Thus, either way, the decreasing parts of row $j$ in $P$ and $e_i(P)$
must contain a box equal to $i$ to the right of $(j,k)$. 
 This means that $(j,k)$ is not the middle position,
 so $P_{jk} = i+1$ and $e_i(P)_{jk} = i$ are both unprimed. Therefore
  $b$ will bump $i$ from row $j$ of both $P$ and $e_i(P)$, though from different columns,
 and then the same number (possibly after adding a prime) will be inserted into the next row 
 for both tableaux. As above, this means that
  the   two insertion processes will
 bump the same entries in every subsequent row, so  
  $\Qqq(e_i(w)) = \Qqq(w)$
 and  
 the primed boxes in $\Pqq(e_i(w)) $ and $\Pqq(w)$ have the same locations.

\item Finally suppose $b$ is greater than $i+1$ but not greater than some entry 
to the left of $(j,k)$ in row $j$ of $P$.
Then $b$ will bump $P_{jk}$ from $P$ and $P_{jk} - 1$ from $e_i(P)$.
On the next step of the insertion algorithm for 
 $w_1 \toQQ P$,
some primed number $c \in \{i+1',i+1\}$ (equal to $P_{jk}$ or to $P_{jk}$ with its prime reversed) will be inserted into the decomposition tableau $T$
composed of the rows of $P$ after row $j$;
while on
 the next step of the insertion algorithm for 
 $w_1 \toQQ e_i(P)$, the number $c-1$ will be inserted into
 the
 same tableau $T$.
 The reverse row reading word of $T$
must have no $i$-unpaired entries equal to $i'$ or $i$
since $P_{jk}$ is $i$-unpaired in $\revrow(P)$.
This situation is therefore identical to one we have already analyzed,
when comparing the effect of inserting $w_1$ and $w_1-1$ into $P$
when $e_i(w) = e_i(w_1)v = (w_1-1)v$. By repeating the argument
in that case, we deduce again that
 $\Qqq(e_i(w)) = \Qqq(w)$ and 
 that the set of primed boxes in $\Pqq(e_i(w))$ and $\Pqq(w)$
are the same.
\eei
 This case analysis completes the proof of the lemma.
\end{proof}

\begin{proposition}\label{01-prop}
Let $w=w_1w_2\cdots w_m \in (\BB_n^+)^{\otimes m} $ 
and $i \in \{0,1,2,\dots,n-1\}$. Then
$ \Pqq(e_i( w)) = e_i (\Pqq(w))$,
and if $e_i(w)\neq0$ then $\Qqq(e_i (w)) = \Qqq(w)$.
\end{proposition}

\begin{proof}
If there are no letters of $w$ equal to $1'$ or $1$
then the same is true of $\revrow(\Pqq(w))$ so $e_0(w) $ and $e_0 (\Pqq(w)) $ are both zero.

Assume  there is some minimal
 $j \in [m]$ with $w_j \in \{1',1\}$.
When $w_j$ is inserted into $\Pqq(w_{j+1}\cdots w_m)$,
it   becomes the new middle element of the first row,
as well as the first entry equal to $1'$ or $1$ in the reverse row reading word.
 Each  subsequent insertion of the numbers $w_{j-1},\dots,w_2,w_1$
(which are all greater than $1$) can only bump $w_j$ 
from its row by changing the row's middle position, which results in $w_j$   
being inserted into the next row, where it becomes the new middle element.
 Since whenever $w_j$ is inserted into a row (on the first iteration or after being bumped)
  the row's middle position changes, 
  all insertions into later rows 
  (as well as whether the insertion is ultimately classified as odd or even) are independent of whether $w_j$ is primed.
  
From these observations, we deduce that 
the number $w_j \in\{1',1\}$ is primed if and only if  the first box of $\Pqq(w)$
containing $1$ or $1'$ in the reverse row reading word order is   primed.
Moreover, 
 toggling the prime on $w_j$ has the effect of toggling the prime on the 
  first box of $\Pqq(w)$
containing $1$ or $1'$ in the reverse row reading word order,
  while preserving $\Qqq(w)$.
We conclude that if $w_j =1$ then $e_0(w) $ and $e_0 (\Pqq(w)) $ are both zero,
and if $w_j = 1'$ then
$ \Pqq(e_0( w)) = e_0 (\Pqq(w))$
and   $\Qqq(e_0 (w)) = \Qqq(w)$.

Now let $i \in [n-1]$. There is nothing to check if the word $w$ is empty so assume $m>0$.
If $e_i(w) = 0$ then  $P(e_i(w)) = 0$ by Lemma~\ref{main-lem1}.
If $e_i(w) \neq 0$ then we can invoke Lemma~\ref{main-lem2} to deduce  
that $ \Pqq(e_i( w)) = e_i (\Pqq(w))$
and   $\Qqq(e_i (w)) = \Qqq(w)$, as we can assume 
the required hypotheses by induction.
\end{proof}

 \begin{lemma}\label{bar-lem1}
 Suppose $P$ is a primed decomposition tableau and $(a,b)$ is the first 
 box in its reverse row reading word order with $P_{ab}= 1^\circ \in \{1',1\}$.
Let $x=2^\bullet \in \{2',2\}$ and $\tilde x = 1^\circ$, and 
 form $\tilde P$ by replacing $P_{ab}$ by $1^\bullet$.
 Then 
\[e_{\bar 1}(x \toQQ P) = \tilde x \toQQ  \tilde P\]
 and the insertions $x\toQQ P$ and $\tilde x \toQQ  \tilde P$
 are both even or both odd.
 \end{lemma}
 
 \begin{proof}
If the increasing part of the first row of $P$ is empty,
then the
 insertions $x\toQQ P$ and $\tilde x \toQQ \tilde P$ 
 will both add a box to the end of the first row.
 Then:
\bei
\item If $a>1$ then the added box will be the new middle position
 for both tableaux, so both insertions will be odd or even
 according to whether the middle element of the first row of $P$ is primed.
 
\item If $a=1$ then $(a,b)$ must be the 
 rightmost position in the first row of $P$.
 In this case, the first row of $x\toQQ P$
 will end in $\ytab{1^\circ  &2}$ (since the middle position has not changed)
  while the first row of $  \tilde x \toQQ \tilde P$ will end in $\ytab{1& 1^\circ}$
  (since the middle position has changed),
 and the insertion will be odd or even according to whether 
 $\bullet$ indicates a prime.
 \eei
Either way, the following properties hold:
\ben
\item[(a)] The first box of $x\toQQ P$ containing $1'$ or $1$ in the 
reverse row reading order remains $(a,b)$.
\item[(b)] The first box of $x\toQQ P$ containing $2$ or $2'$ in the 
reverse row reading order  is the first box of $\tilde x \toQQ \tilde P$
containing $1'$ or $1$.
\item[(c)] The box in (a) is primed in $x\toQQ P$ (respectively, $\tilde x\toQQ \tilde P$)
 if and only if the box in (b) is primed in $\tilde x\toQQ \tilde P$ (respectively, $x\toQQ P$), and
all other boxes of both tableaux have the same entries.
\item[(d)] The insertions $x\toQQ P$ and $\tilde x \toQQ   \tilde P$
 are both even or both odd.
\een
The first three properties imply that $e_{\bar 1}(x \toQQ P) = \tilde x \toQQ   \tilde P$
as desired.
 
Now assume the increasing part of the first row of $P$ is nonempty.
The first row of $\tilde P$ has the same increasing part
as $P$, and   
 all entries in this part are greater than one.
 Therefore,
both insertions $x\toQQ P$ and $\tilde x \toQQ \tilde P$
will bump the first entry of this sequence, which will then bump some entry from the weakly decreasing part of the row. We again have two cases:
\bei
\item If $a>1$ then  $x=2^\bullet$ and $\tilde x=1^\circ$ will become the respective new middle positions of the first row
of $x\toQQ P$ and $\tilde x \toQQ \tilde P$. Then
the same (possibly primed) number will be inserted into the next row,
and for both insertions all subsequent steps will proceed in exactly the same way,
except that if $1^\circ$ is bumped from $(a,b)$ when inserting $x\toQQ P$
then $1^\bullet$ will be bumped from $(a,b)$ when inserting $\tilde x \toQQ \tilde P$.
When this happens, $1^\circ $ and $1^\bullet$
will be inserted into row $a+1$ of $P$ and $\tilde P$, where they will bump the same numbers to 
become the new middle entries,
and in all remaining rows the two insertion processes will proceed in parallel.

\item If $a=1$ then $(a,b)$ must be the middle position of the first row of both $P$ and $\tilde P$. If this box is bumped when inserting $x\toQQ P$  then it is also bumped
 when inserting $\tilde x\toQQ \tilde P$ (since whether this
 happens depends only on other entries in the first row of $P$
 that are the same in $\tilde P$), and all observations in the previous case will still apply.
Otherwise, $(a,b)$ will remain the middle  box when inserting $x\toQQ P$ but not when inserting $\tilde x\toQQ \tilde P$. 
This means that the first row of $x\toQQ P$ will have the form 
$\ytab{\cdots & 1^\circ & 2& \cdots}$ and a number of the form $y^\bullet$
will be inserted into the next row,
while the first row of $\tilde x \toQQ \tilde P$ will have the form 
$\ytab{\cdots & 1 & 1^\circ& \cdots}$ and the same number $y^\bullet$ will be inserted into the next row.
Thus both insertions will again proceed in exactly the same way after the first row.
\eei
In these cases   properties (a), (b), (c), and (d) above are all still true,
so we again have $e_{\bar 1}(x \toQQ P) = \tilde x \toQQ  \tilde P$.
  \end{proof}

  \begin{lemma}\label{bar-lem2}
 Suppose $P$ is a primed decomposition tableau 
with at least one entry in  $\{1',1\}$ and at least one entry in $\{2',2\}$ such that $e_{\bar 1}(P)\neq 0$.
Let $x$ be any primed number greater than $2$. Then 
$e_{\bar 1}(x \toQQ P) =x\toQQ e_{\bar 1}(P)$
and the insertions $x\toQQ P$ and $x \toQQ e_{\bar 1}(P)$ are both even or both odd.
 \end{lemma}
 
 \begin{proof}
 Suppose the first box of $P$ containing $1'$ or $1$ 
 in the reverse row reading word order is $(a_1,b_1)$
 and the first box of $P$ containing $2'$ or $2$ in the reverse row reading word
 order if $(a_2,b_2)$.
 Since $e_{\bar 1}(P)\neq0$, we either have $a_1>a_2$,
 in which case both $(a_1,b_1)$ and $(a_2,b_2)$ are the middle positions in their rows,
 or $a_1=a_2$, in which case $(a_1,b_1)$ is the middle position in its row
 and $b_1 + 1 = b_2$.
 
Recall that if $P_{a_1b_1} = 1^\circ$ and $P_{a_2b_2} = 2^\bullet$
 then $e_{\bar 1}(P)$ is formed from $P$
 by changing the entries just named to $1^\bullet$ and $1^\circ$ respectively.
 When $a_1=a_2$, this operation moves the middle position in row $a_1$ one column to the right, but otherwise all middle positions in $P$ and $e_{\bar 1}(P)$ are the same.

 Let $T$ and $\tilde T$ 
 be the sub-tableaux composed of the rows of $P$ 
and $e_{\bar 1}(P)$ after row $a_2$.
The insertions $x\toQQ P$ and $x \toQQ e_{\bar 1}(P)$
 will proceed in exactly the same way until row $a_2$,
 where some number $y>2$ will be inserted in both cases.
 
  First assume $a_1>a_2$. 
 Then $(a_2,b_2)$ is the middle position in its row in both $P$ and $e_{\bar 1}(P)$,
 so it will be bumped in both insertions or in neither:
 \bei
 \item
 In the bumped case, the first $a_2$ rows of $x\toQQ P$ and $x \toQQ e_{\bar 1}(P)$
 will be identical, while the remaining rows 
 of $x\toQQ P$ and $x \toQQ e_{\bar 1}(P)$ will be $2^\bullet \toQQ T$ 
 and $1^\circ \toQQ \tilde T$, respectively.
Comparing the latter tableaux is exactly the situation of Lemma~\ref{bar-lem1},
which implies that $e_{\bar 1}(2^\bullet \toQQ T) = 1^\circ \toQQ \tilde T$.
Since the first $a_2$ rows of $x\toQQ P$ and $x \toQQ e_{\bar 1}(P)$
contain no entries in $\{1',1,2',2\}$,
we deduce that $e_{\bar 1}(x \toQQ P) =x\toQQ e_{\bar 1}(P)$
and the insertions $x\toQQ P$ and $x \toQQ e_{\bar 1}(P)$ are both even or both odd.
\item
Assume instead that $(a_2,b_2)$ is not bumped in either insertion.
Then the middle positions of row $a_2$ in $P$ and $e_{\bar 1}(x)$
 will not change, and the same number $z$ will be inserted
 into rows $a_2+1$ of both $P$ and $e_{\bar 1}(P)$.
 In this case the first $a_2$ rows of  
 $x\toQQ P$ and $x \toQQ e_{\bar 1}(P)$
 will
 only differ in box $(a_2,b_2)$, which is the first box in  $x\toQQ P$
with an entry in $\{2',2\}$ and the first box in $x\toQQ e_{\bar 1}(P)$
with an entry in $\{1',1\}$.
Moreover,
 the rows of $P$ and $e_{\bar 1}(P)$ after row $a_2$
will just be $z \toQQ T$ and $z \toQQ\tilde T$, respectively.
It follows as in Proposition~\ref{01-prop}
that $y \toQQ T$ and $y \toQQ\tilde T$ will be identical
except that the first tableau will have $1^\circ$ while the other tableau will have $1^\bullet$ in the first box with an entry in $\{1',1\}$ in the reverse row reading order.
Therefore, we again have $e_{\bar 1}(x \toQQ P) =x\toQQ e_{\bar 1}(P)$
and the two insertions are both even or both odd.
\eei
 
Now assume $a_1=a_2$.  Then  $P_{a_2b_2} = 2^\bullet = 2$
and $e_{\bar 1}(P)_{a_1b_1} = 1^\bullet = 1$ since these entries are not in the middle boxes of the row.
The inserted number $y>2$ cannot bump position $(a_2,b_2) = (a_1,b_1+1)$
from $P$ or from $e_{\bar 1}(P)$, since in the first tableau this box belongs to the increasing part of the row,
and in the second tableau it is part of the weakly decreasing part and has the same entry as the box to its left when primes are ignored.
Therefore, if $y$ bumps an entry from the increasing part of row $a_1$ in either tableaux,
then it will bump the same primed number $z^\ast>2$ from both
(here $z \in \PP$ and $\ast$ denotes another arbitrary, unspecified choice of prime),
and this number will
cause
 $(a_1,b_1)$ to be bumped from the weakly decreasing part   in both insertions or in neither:
  \bei
  
 \item
 Suppose  $(a_1,b_1)$ is bumped in both insertions.
 Then for $x\toQQ P$,
  the middle entry in row $a_1$ 
will move from column $b_1$ to column $b_2=b_1+1$ and change from $1^\circ$ to $2^\ast$,
and then $1^\circ$ will be inserted into row $a_1+1$.
But 
for $x\toQQ e_{\bar 1}(P)$,
the middle entry in row $a_1$ 
will remain $1^\circ$ in column $b_2$,
and then $1^\ast $ will be inserted into row $a_1+1$.
We have $T=\tilde T$ as $a_1=a_2$, and 
it follows as in Proposition~\ref{01-prop}
that $1^\circ \toQQ T$ 
and $1^\ast \toQQ T$ are identical
except that the first tableau has $1^\circ$ while the other tableau has $1^\ast$ in the first box with an entry in $\{1',1\}$ in the reverse row reading order.
As $(a_2,b_2)$ remains the first box of $x\toQQ P$ containing $2'$ or $2$
and the first box of $x\toQQ e_{\bar 1}(P)$ containing $1'$ or $1$,
 we have 
  $e_{\bar 1}(x \toQQ P) =x\toQQ e_{\bar 1}(P)$
and the two insertions are both even or both odd.

 \item Otherwise, $x \toQQ P$ and $x\toQQ e_{\bar 1}(P)$
 will be identical outside boxes $(a_1,b_1)$ and $(a_2,b_2)$ and the two insertions will   bump the same entries on all iterations.
As the boxes of $x \toQQ P$ before $(a_2,b_2)$ in the reverse row reading word order contain only entries greater than $2$, 
 $e_{\bar 1}(x \toQQ P) =x\toQQ e_{\bar 1}(P)$
and the two insertions are both even or both odd.
\eei
This case analysis completes the proof of the lemma.
  \end{proof}

\begin{proposition}\label{bar1-prop}
Let $w=w_1w_2\cdots w_m \in (\BB_n^+)^{\otimes m} $.
Then
$ \Pqq(e_{\bar 1}( w)) = e_{\bar 1} (\Pqq(w))$,
and if $e_{\bar 1}(w)\neq0$ then $\Qqq(e_{\bar 1}(w)) = \Qqq(w)$.
\end{proposition}

\begin{proof}
Let $j,k\in[m]$ be minimal with $w_j \in \{2',2\}$ and $w_k \in \{1',1\}$.
 If $j$ does not exist 
 or $j>k$, then $e_{\bar 1}(w) = 0$ so $P(e_{\bar 1}(w)) = 0$ by Lemma~\ref{main-lem1}.

If $j$ exists but $w$ has no letters   equal to $1'$ or $1$ (so that $k$ is undefined),
then the same is true of the reverse row reading word of $\Pqq(w)$.
In this case $e_{\bar 1}$ has the same effect on $w$ and $\Pqq(w)$
as $e_1$,
so we have
\[\label{1b-check-eq1}
 \Pqq(e_{\bar 1}(w))
=
\Pqq(e_{ 1}(w))
=
e_1(\Pqq(w))
 = e_{\bar 1}(\Pqq(w)) 
\]
and
$
\Qqq(e_{\bar 1}(w)) = 
\Qqq(e_{  1}(w)) =
 \Qqq(w)
 $ as desired.

 Finally suppose $j$ and $k$ are both defined with $j<k$.
 Let $\tilde w = e_{\bar 1}(w)$ so that if $w_j = 2^\circ$ and $w_k = 1^\bullet$
 (using the superscripts to denote an arbitrary, unspecified choices of primes)
 then $\tilde w_j = 1^\bullet$ and $\tilde w_k = 1^\circ$.
 By Proposition~\ref{01-prop}
 the tableaux $P:=\Pqq(\tilde w_{j+1} \cdots \tilde w_m)$ and 
 $\tilde P:= \Pqq(w_{j+1}\cdots w_m)$ 
 are related as in Lemma~\ref{bar-lem1},
 and we have $\Qqq( w_{j+1} \cdots  w_m)=\Qqq( \tilde w_{j+1} \cdots  \tilde w_m)$, 
 so  Lemma~\ref{bar-lem1} implies that
\[
e_{\bar 1}(\Pqq(w_j\cdots w_m)) = e_{\bar 1}(w_j \toQQ P) = \tilde w_j \toQQ \tilde P = 
\Pqq(\tilde w_j \cdots \tilde w_m)
\]
and $\Qqq( w_{j} \cdots  w_m)=\Qqq( \tilde w_{j} \cdots  \tilde w_m)$.
 Then, by iterating
 Lemma~\ref{bar-lem2}, we deduce that we have 
 \[\ba
e_{\bar 1}(\Pqq(w)) &= 
w_1 \toQQ \cdots \toQQ w_{j-1}\toQQ e_{\bar 1}(\Pqq(w_j\cdots w_m)) 
\\&=
\tilde w_1 \toQQ \cdots \toQQ \tilde w_{j-1}\toQQ 
\Pqq(\tilde w_j \cdots \tilde w_m) = \Pqq(e_{\bar 1}(w))
\ea\] 
and $\Qqq(e_{\bar 1}(w)) = \Qqq(w)$, as desired.
\end{proof}

\begin{proof}[Proof of Theorem~\ref{main-thm}]
Choose 
an index $i \in \{\bar 1,0,1,2,\dots,n-1\}$
and
a word $w \in (\BB_n^+)^{\otimes m}$.
By Propositions~\ref{01-prop} and \ref{bar1-prop}
we have $\Pqq(e_i(w)) = e_i(\Pqq(w))$,
and if $e_i(w) \neq 0$
then $\Qqq(e_i(w)) = \Qqq(w)$.
Thus, if $f_i(w) \neq 0$ then \[ e_i(\Pqq(f_i(w)))=\Pqq(e_i(f_i(w)))= \Pqq(w)\]
so   $\Pqq(f_i(w)) =f_i(\Pqq(w))$.
If  $f_i(\Pqq(w)) \neq 0$
then  some  $v \in (\BB_n^+)^{\otimes m}$
has \[\Pqq(v) = f_i(\Pqq(w))\quand \Qqq(v) = \Qqq(w)\] by Proposition~\ref{qq-bij-thm}.
Since for this word we have both
\[\Pqq(e_i(v)) = e_i(\Pqq(v)) = e_i(f_i(\Pqq(w))) = \Pqq(w)
\]
and
$ \Qqq(e_i(v)) = \Qqq(v) =\Qqq(w)$,
it follows by same theorem
that
 $e_i(v) =w$, and so $f_i(w) =v \neq 0$.
 Taking contrapositives, we conclude that if $f_i(w) = 0$
 then $f_i(\Pqq(w)) = 0$.
 Thus, more generally, we have
 $\Pqq(f_i(w)) = f_i(\Pqq(w))$, and if $f_i(w) \neq 0$ then $e_i(f_i(w)) = w \neq 0$ so
 \[\Qqq(f_i(w)) =\Qqq(e_i(f_i(w)))= \Qqq(w).\]

This shows that the map $w \mapsto \Pqq(w)$ in Theorem~\ref{main-thm},
which is evidently weight-preserving,
 also commutes with all $\q_n$ and $\qq_n$-crystal operators. Moreover, 
 the recording tableau $\Qqq$ is constant on the connected components
of 
$\BB_n^{\otimes m}$ and $ (\BB_n^+)^{\otimes m}$.
The set of all words $w$ in 
 $\BB_n^{\otimes m}$ or $ (\BB_n^+)^{\otimes m}$ with the same fixed recording tableau
 $Q = \Qqq(w)$
 is therefore a union of full subcrystals;
 these sets are actually connected as their images under $w \mapsto \Pqq(w)$
and
 connected crystals of the form $\DTab_n(\lambda)$
or $\DDTab_n(\lambda)$.

 It remains to show that the unions in Theorem~\ref{main-thm}
 are over the right sets of strict partitions; that is, we must explain why if $ w \in (\BB_n^+)^{\otimes m}$
 then $\Pqq(w)$ has at most $n$ rows (as it has $m$ boxes by definition).
 This holds by Lemma~\ref{nonempty-lem} since $\unprime(\Pqq(w))$ is a decomposition tableau.
\end{proof}

\subsection{Applications to normal crystals}\label{normal-sect}

Following \cite{GHPS,Marberg2019b,MT2023}, we define
a ($\gl_n$-, $\q_n$-, or $\qq_n$-) crystal to be \defn{normal} if each of its full subcrystals is isomorphic to a full subcrystal of a tensor power of the relevant standard crystal, as indicated in Examples~\ref{st-ex1}, \ref{st-ex2}, and \ref{st-ex3}.
In this definition we interpret the $0$th tensor of the standard crystal as the crystal 
that has a single element whose weight is $0 \in \ZZ^n$.

Normal $\gl_n$-crystals are sometimes called \defn{Stembridge crystals}, since they are characterized by the local \defn{Stembridge axioms} \cite{Stembridge2003}.
In all types, normal crystals are seminormal and preserved by 
disjoint unions and tensor products.

\begin{corollary}\label{thm-cor1}
If  $\lambda$ is a strict partition with at most $n$ parts, then the $\qq_n$-crystal
$\DDTab_n(\lambda)$ is connected and normal with highest weight $\lambda$.
\end{corollary}
\begin{proof}
These properties are immediate from Theorems~\ref{revrow-thm} and \ref{highest-thm}.
\end{proof}

One motivation for the results in this article was to provide a simpler
and more intuitive proof of the following theorem, which was our main result in \cite{MT2023}.

\begin{theorem}[\cite{MT2023}]\label{normal-thm}
A connected normal $\qq_n$-crystal
has a unique $\qq_n$-highest weight element, whose
 weight    is a  strict partition $\lambda$ with at most $n$ parts. 
Conversely, for each strict partition $\lambda$ with at most $n$ parts,
there is a connected normal $\qq_n$-crystal with highest weight $\lambda$.
Finally, there is a unique isomorphism between any two connected normal $\qq_n$-crystals with the same highest weight.
\end{theorem}

\begin{proof}
Suppose $\cB$ is a connected normal $\qq_n$-crystal. Then $\cB \cong \DDTab_n(\lambda)$
for some strict partition $\lambda$ with $\ell(\lambda)\leq n$ by Theorem~\ref{main-thm}.
Therefore $\cB$ has a unique $\qq_n$-highest weight element by Theorem~\ref{highest-thm},
and the weight of this element is $\lambda$.
The second assertion in the theorem is just Corollary~\ref{thm-cor1}. The last claim holds since 
there is at most one isomorphism between any two crystals with unique highest weight elements.
\end{proof}

By essentially the same argument, one can derive a $\q_n$-version of this theorem
(see, for example, \cite[Thm.~1.3]{MT2023}); this proof strategy is similar to what appears in \cite{GJKKK}.
There is also a classical version of Theorem~\ref{normal-thm}
for normal $\gl_n$-crystals;
 see \cite[Thms.~3.2 and 8.6]{BumpSchilling}.
The existence of unique highest weight elements guaranteed by these theorems implies the following fundamental property:

\begin{corollary}\label{atmost-cor}
Fix $\g \in \{\gl_n, \q_n, \qq_n\}$ and suppose $\cB$ and $\cC$ are connected normal $\g$-crystals.
Then there is at most one $\g$-crystal isomorphism $\cB \xrightarrow{\sim} \cC$.
\end{corollary}

Recall that $\revrow$ is the reverse reading word of a tableau.
  Since the composition $\Pqq \circ \revrow :  \DDTab_n(\lambda) \xrightarrow{\sim} \DDTab_n(\lambda)$  
  is a $\qq_n$-isomorphism by Theorems~\ref{revrow-thm} and \ref{main-thm}, it must coincide with the identity map. Therefore: 
  
  \begin{corollary}\label{P-revrow-cor}
If $T \in \DDTab_n(\lambda)$ then $\Pqq(\revrow(T)) = T$.
\end{corollary}

In our previous work \cite[Thm~7.16]{MT2023}, we identified a connected normal $\qq_n$-crystal with
unique highest weight $\lambda$ on the set of semistandard shifted tableaux $\SShTab_n(\lambda)$ (allowing diagonal primes). This extends a connected normal $\q_n$-crystal structure on $\ShTab_n(\lambda)$ (excluding diagonal primes) studied in \cite{AssafOguz,HPS,Hiroshima2018}.
There must be an isomorphism $\DDTab_n(\lambda) \xrightarrow{\sim} \SShTab_n(\lambda)$ by the previous theorem.

We can identify this isomorphism, though we must outsource the key definitions.
For a primed word $w$, define its \defn{mixed insertion tableau} $\PHM(w)$ as in \cite[Def.~5.17]{M2021a}.
This is a semistandard shifted tableau by \cite[Cor.~5.22]{M2021a}.

\begin{proposition}\label{normal-prop1}
Suppose  $\lambda$ is a strict partition with at most $n$ parts.
Then $ \PHM\circ \revrow$
is a $\qq_n$-crystal isomorphism
$  \DDTab_n(\lambda) \to \SShTab_n(\lambda)$
 which restricts to a $\q_n$-crystal isomorphism
$\DTab_n(\lambda) \to \ShTab_n(\lambda)$.
\end{proposition}

\begin{proof}
The map $\revrow$ is a crystal embedding $\DDTab_n(\lambda) \to (\BB_n^+)^{\otimes |\lambda|}$ by Theorem~\ref{revrow-thm},
while $\PHM : (\BB_n^+)^{\otimes |\lambda|} \to \bigsqcup_{\mu}  \SShTab_n(\mu)$ (the union over $\mu\vdash|\lambda|$ strict with $\ell(\mu) \leq n$)
is a quasi-isomorphism by \cite[Cor.~7.13]{MT2023}.
Their composition is therefore a quasi-isomorphism, which must restrict to an isomorphism $  \DDTab_n(\lambda) \xrightarrow{\sim} \SShTab_n(\lambda)$ since $\SShTab_n(\mu)$ has highest weight $\mu$ by \cite[Thm.~6.20]{MT2023}.

When restricted to unprimed words, 
 $\PHM$ coincides with Haiman's original definition of \defn{shifted mixed insertion},
and the same argument using \cite[Thm.-Def.~2.12]{Marberg2019b} instead of \cite[Thm.~6.20]{MT2023} shows that $ \PHM\circ \revrow$ is also a $\q_n$-crystal isomorphism
$\DTab_n(\lambda) \xrightarrow{\sim} \ShTab_n(\lambda)$.
\end{proof}

The previous result implies that 
\be
\ba\ch(\DTab_n(\lambda)) &=\ch(\ShTab_n(\lambda))= P_\lambda(x_1,\dots,x_n),\\
\ch(\DDTab_n(\lambda)) &= \ch(\SShTab_n(\lambda))=Q_\lambda(x_1,\dots,x_n).\ea
\ee
Taking the limit as $n\to\infty$ produces the identities
\be\textstyle
P_\lambda = \sum_{T \in \DTab(\lambda)} x^T
\quand
Q_\lambda = \sum_{T \in \DDTab(\lambda)} x^T\ee
as well as the following statement.

\begin{corollary}\label{normal-cor3}
If $\lambda$ is a strict partition then  $\PHM\circ \revrow $ is a weight-preserving bijection
$
\DTab(\lambda) \to \ShTab(\lambda) $ and $ \DDTab(\lambda) \to \SShTab(\lambda)$.
\end{corollary}

The facts above show that each connected component of a normal $\q_n$- or $\qq_n$-crystal
is respectively isomorphic to $\DDTab_n(\lambda)$ or $\DTab_n(\lambda)$ for some strict partition $\lambda$ with at most $n$ parts. As the Schur $P$- and $Q$-polynomials in $n$ variables indexed by such partitions are linearly independent over $\ZZ$ \cite[\S{III.8}]{Macdonald},
we recover the following results from \cite{GJKKK} (for $\q_n$) and \cite{MT2023} (for $\qq_n$).
 
\begin{corollary}\label{normal-cor2}
The character of a finite normal $\qq_n$-crystal (respectively, $\q_n$-crystal) is Schur $Q$-positive (respectively, Schur $P$-positive), and 
two such crystals with the same character are isomorphic.
\end{corollary}

We conclude this section with a comment about highest and lowest weights.

 \begin{proposition}\label{ranked-prop}
Suppose $\cB$ is a normal $\gl_n$, $\q_n$, or $\qq_n$-crystal.
Let $[n-1] \subseteq I\subseteq \{\bar 1,0\}\sqcup[n-1]$ be the relevant indexing set for the crystal operators.
Then there is a unique function $\rank : \cB \to \NN$ such that 
(a) $\rank(f_i(b)) = \rank(b) + 1$ if $f_i(b) \neq 0$ for all $i \in I$ and $b \in \cB$, and
(b) $\rank(b) = 0$ for some $b$ in each connected component of $\cB$.
Relative to this map, an element $b \in \cB$ is highest weight if and only if $\rank(b) =0$,
and  lowest weight if and only if $\rank(b)$ is the maximum value attained 
by any element in the connected component of $b$.
\end{proposition}

This result means we can read off the highest weight elements of normal crystals just from
the usual crystal graph, without drawing any extra arrows.

\begin{proof}
Any rank function on $\cB$ satisfying (a) is unique up to translation by a constant on each connected component.
So it suffices to produce one such function that is also a rank function for the extended crystal graph in types $\q_n$ and $\qq_n$
(which includes extra arrows for the operators $f_{\bar i}$ and $f_0^{[i]}$).
For this rank function to take nonnegative values and also attain the value zero on each connected component,
it must have $\rank(b) =0$ if and only if $b$ is highest weight,
and it must attain its local maximum when $b$ is lowest weight.

If $v \in \ZZ^n$ has $v_1+v_2+\dots+v_n =0$ then we can uniquely decompose $v = \sum_{i=1}^{n-1} c_i (\e_i -\e_{i+1})$
for $c_i \in \ZZ$; in this case define $\height(v) := \sum_{i=1}^{n-1} c_i$.
Suppose $b \in \cB$ belongs to a connected component with unique highest weight element $b_0$.
In types $\gl_n$ and $\q_n$ the desired rank function
is given by $\rank(b) = \height(\weight(b_0) - \weight(b))$, since every crystal operator $f_i$ or $f_{\bar i}$ for $i \in[n-1]$
subtracts $\e_i - \e_{i+1}$ from the weight.

In type $\qq_n$ we may assume that $\cB$ is a disjoint union of crystals of the form $\DDTab_n(\lambda)$.
If $\primes(b)$ is the number of primed boxes in $b$ then
 the desired rank function is $\rank(b) := \height(\weight(b_0) - \weight(b)) +\primes(b)$, 
 since each $f_0^{[i]}$ increases the number of primed boxes by one.
\end{proof}

\subsection{Shifted plactic relations for primed words}\label{comp-sect}

In this section we study the relation $\simdec$ on primed words with the property that $v\simdec w$ if and only if $\Pqq(v) =\Pqq(w)$. This will generalize the notion of \defn{shifted plactic equivalence} in \cite{CNO,Serrano}.

We use the term \defn{congruence} to mean an equivalence relation $\sim$ on primed words that has 
$ab \sim uv$ whenever $a,b,u,v$ are words with $a\sim u$ and $b\sim v$.

\begin{definition}\label{simdec}
Let $\simdec$ be the smallest congruence that satisfies
\begin{align}
  a^\bullet\hs   b &\ \simdec\    a^\bullet\hs    b' & \quad \text{if} \quad a\leq b, \label{simdec-1} \\
  b\hs a^\bullet &\ \simdec\    b'\hs a^\bullet & \quad \text{if} \quad a< b, \label{simdec-2} \\
%  a\hs   b &\ \simdec\    a\hs    b' & \quad \text{if} \quad a\leq b,   \\
%   a'\hs   b &\ \simdec\    a'\hs    b' & \quad \text{if} \quad a\leq b,   \\
%  b\hs a &\ \simdec\    b'\hs a & \quad \text{if} \quad a< b,\\
%  b\hs a' &\ \simdec\    b'\hs a' & \quad \text{if} \quad a< b,\\
a^\bullet\hs    b^{\phantom{\bullet}}\hs d^{\phantom{\bullet}}\hs c^\circ &\ \simdec\  a^\bullet \hs   d^{\phantom{\bullet}}\hs b^\circ\hs    c & \quad \text{if} \quad a\leq b \leq c < d, \label{simdec-3}\\
a^\bullet\hs   c^{\phantom{\bullet}}\hs d^{\phantom{\bullet}}\hs b^\circ &\ \simdec\  a^\bullet\hs    c^{\phantom{\bullet}}\hs b^\circ\hs   d & \quad \text{if} \quad a\leq b < c \leq d, \label{simdec-4} \\
d^{\phantom{\bullet}}\hs a^\bullet\hs   c^{\phantom{\bullet}}\hs b^\circ &\ \simdec\  a^\bullet\hs   d^{\phantom{\bullet}}\hs c^{\phantom{\bullet}}\hs b^\circ  & \quad \text{if} \quad a\leq b < c < d, \label{simdec-5}\\
b^{\phantom{\bullet}}\hs a^\bullet\hs    d^{\phantom{\bullet}}\hs c^\circ &\ \simdec\  b^\circ\hs    d^{\phantom{\bullet}}\hs a^\bullet\hs    c & \quad \text{if} \quad a< b \leq c < d, \label{simdec-6}\\
c^{\phantom{\bullet}}\hs b^\bullet \hs   d^{\phantom{\bullet}}\hs a^\circ &\ \simdec\  c^\bullet \hs   d^{\phantom{\bullet}}\hs b^{\phantom{\bullet}}\hs a^\circ  & \quad \text{if} \quad a< b < c \leq d, \label{simdec-7}\\
d^{\phantom{\bullet}}\hs b^\bullet\hs   c^{\phantom{\bullet}}\hs a^\circ &\ \simdec\  b^\bullet\hs    d^{\phantom{\bullet}}\hs c^{\phantom{\bullet}}\hs a^\circ   & \quad \text{if} \quad a< b \leq c < d, \label{simdec-8}\\
b^\bullet \hs   c^{\phantom{\bullet}}\hs d^{\phantom{\bullet}}\hs a^\circ &\ \simdec\  b^\bullet\hs   c^{\phantom{\bullet}}\hs a^\circ\hs    d & \quad \text{if} \quad a< b \leq c \leq d, \label{simdec-9}\\
c^{\phantom{\bullet}}\hs a^\bullet \hs   d^{\phantom{\bullet}}\hs b^\circ &\ \simdec\  c^\circ\hs   d^{\phantom{\bullet}}\hs a^\bullet \hs   b & \quad \text{if} \quad a\leq b < c \leq d, \label{simdec-10} 
\end{align}
for all unprimed integers $a,b,c,d \in \PP$
and all choices of $a^\bullet,a^\circ \in \{a',a\}$, 
$b^\bullet,b^\circ \in \{b',b\}$, 
and
$c^\bullet,c^\circ \in \{c',c\}$
with
$a-a^\bullet =b-b^\bullet=c-c^\bullet$ and
$a-a^\circ =b-b^\circ=c-c^\circ$.
\end{definition}

One could rewrite this definition without using the annotated symbols $a^\bullet$, $a^\circ$, $b^\bullet$, $b^\circ$, $c^\bullet$, $c^\circ$ but this would require 36 relations instead of 10.
For example \eqref{simdec-1} could be 
rewritten as the pair of relations $ab\simdec ab'$ and $a'b\simdec ab'$ for unprimed numbers $a\leq b$ while \eqref{simdec-3} could be rewritten as the four relations 
\[ abdc \simdec adbc,\quad
a'bdc \simdec a'dbc,\quad abdc' \simdec adb'c,\quand 
a'bdc'\simdec a'db'c
\]
for unprimed numbers $a\leq b \leq c <d$.

\begin{example}\label{simdec-ex0} We have
\[
\ba
16431'224 & \simdec 61431'224 &&\quad\text{since \eqref{simdec-5} gives $1643\simdec 6143$}
\\& \simdec 64131'224 &&\quad\text{since \eqref{simdec-5} gives $1431'\simdec 4131'$}
\\& \simdec 641132'24 &&\quad\text{since \eqref{simdec-3} gives $131'2\simdec 1132'$}
\\& \simdec 6411232'4 &&\quad\text{since \eqref{simdec-3} gives $132'2\simdec 1232'$}
\\& \simdec 64112342' &&\quad\text{since \eqref{simdec-4} gives $232'4\simdec 2342'$}.
\ea
\]
\end{example}

Clearly if two primed words have $v\simdec w$ then $\unprime(v) \simdec \unprime(w)$. 
It follows that $\simdec$ restricts on unprimed words to the \defn{shifted plactic equivalence relation} specified in \cite[Def.~1.6]{Serrano}.

We first state two lemmas relevant to the proof of Proposition~\ref{simdec-prop}.
 In these statements, we annotate certain letters using the symbols $\bullet$ and $\circ$ following the same conventions as in Definition~\ref{simdec}.
 
\begin{lemma}\label{simdec-incr}
Let $y_0 < y_1 < y_2 \dots < y_N$ be an increasing sequence in $\PP$ for some $N \geq 2$, and fix $x \in \PP$. If $y_{i-1}<x \leq y_{i}$ for some $i \geq 2$, then
\begin{align*}
x^\bullet (y_Ny_{N-1}\dots y_1y_0^\circ) \simdec y_Ny_{N-1}\dots y_{i+1} x y_{i-1}\dots  y_{1}^\bullet y_{i}y_0^\circ .
\end{align*}
If instead $x \leq y_1$ then
$
x^\bullet (y_Ny_{N-1}\dots y_1y_0^\circ) \simdec y_Ny_{N-1}\dots y_{2} x^\bullet y_1y_0^\circ.
$
\end{lemma}

\begin{proof}
If $N =2$ and $y_1< x \leq y_{2}$, then 
\[ x^\bullet y_2 y_1y_0^\circ \simdec x y_{1}^\bullet y_2 y_0^\circ \quad \text{ by } \eqref{simdec-7},
\]
while if $x \leq y_{1}$ then
\[ x^\bullet y_2 y_1y_0^\circ \simdec  y_2 x^\bullet y_1y_0^\circ \quad   \text{ by } \eqref{simdec-5} \text{ or } \eqref{simdec-8}.
\]
Suppose $N \geq 3$. If $y_{N-1}< x \leq y_{N}$, then 
\begin{align*}
x^\bullet (y_Ny_{N-1}y_{N-2} \dots y_0^\circ) & \simdec x(y_{N-1}^\bullet y_Ny_{N-2} y_{N-3}) \dots y_0^\circ \quad  \text{ by } \eqref{simdec-7} \\
& \simdec xy_{N-1}y_{N-2}^\bullet y_N y_{N-3} \dots y_0^\circ \quad  \text{ by } \eqref{simdec-7}\\
& \simdec \dots\\
& \simdec  xy_{N-1}\dots y_2 y_{1}^\bullet y_Ny_0^\circ \quad  \text{ by } \eqref{simdec-7}.
\end{align*}
Finally if $y_{i-1}<x \leq y_{i}$ for some $N>i \geq 2$, then
\begin{align*}
x^\bullet (y_Ny_{N-1}y_{N-2} \dots y_0^\circ) & \simdec y_Nx^\bullet  y_{N-1}y_{N-2} \dots y_0^\circ \text{ by } \eqref{simdec-5} \text{ or } \eqref{simdec-8},
\end{align*}
and we obtain the desired result by induction on $N$.
\end{proof}

\begin{lemma}\label{simdec-decr}
Let $w_1 \geq w_2 \geq \dots \geq w_m$ be a weakly decreasing sequence in $\PP$ for some $m \geq 1$ and suppose $y \in \PP$ has $w_m < y$. Fix $u \in \PP$ with $u \leq y$. Then: 
\ben
\item[(i)] If $u > w_m$ and $w_1< y$, then
\begin{align*}
u^\bullet y w_m^\circ w_{m-1} \dots w_1 \simdec u w_m^\circ w_{m-1} \dots w_2 y  w_1^\bullet.
\end{align*}
\item[(ii)] If $u \leq w_m$ and $w_1< y$, then
\begin{align*}
u^\bullet y w_m^\circ w_{m-1} \dots w_1 \simdec u^\bullet w_m w_{m-1} \dots w_2 y  w_1^\circ.
\end{align*}
\item[(iii)] If $u > w_m$ and $w_j < y \leq w_{j-1}$ for some $2 \leq j < m$, then
\begin{align*}
u^\bullet y w_m^\circ w_{m-1} \dots w_1 \simdec u w_m^\circ w_{m-1} \dots w_{j+1}y w_{j-1} \dots w_1w_j^\bullet.
\end{align*}
\item[(iv)] If $u \leq w_m$ and $w_j < y \leq w_{j-1}$ for some $2 \leq j < m$, then
\begin{align*}
u^\bullet y w_m^\circ w_{m-1} \dots w_1 \simdec u^\bullet w_m w_{m-1} \dots w_{j+1}y w_{j-1} \dots w_1w_j^\circ.
\end{align*}
\item[(v)] If $w_m< y \leq w_{m-1}$, then
\begin{align*}
u^\bullet y w_m^\circ w_{m-1} \dots w_1 \simdec u^\bullet y w_{m-1} w_{m-2} \dots w_1w_m^\circ.
\end{align*}
\een
\end{lemma}
\begin{proof} We consider each part in turn. 
\ben
\item[(i)] In this case $y \geq u > w_m$ and $w_m \leq w_{m-1} \leq w_1 < y$. Therefore
{\footnotesize\begin{align*}
(u^\bullet y w_m^\circ w_{m-1}) \dots w_1 &\simdec u  (w_m^\circ y w_{m-1}^\bullet w _{m-2}) \dots w_1 \text{ by } \eqref{simdec-6} \text{ or } \eqref{simdec-10} \\
& \simdec u  w_m^\circ  (w_{m-1} y w _{m-2}^\bullet w_{m-3}) \dots w_1 \text{ by } \eqref{simdec-3}\\
& \simdec u  w_m^\circ  w_{m-1} ( w _{m-2} y w_{m-3}^\bullet w_{m-4} )\dots w_1 \text{ by } \eqref{simdec-3}\\
& \simdec \dots\\
& \simdec u w_m^\circ w_{m-1} \dots w_3 w_2 y w_1^\bullet \text{ by } \eqref{simdec-3}.
\end{align*}}
\item[(ii)] This case is similar to (i), but now  $u \leq w_m \leq w_{m-1} \leq w_1 < y$. Therefore
{\footnotesize\begin{align*}
(u^\bullet y w_m^\circ w_{m-1}) \dots w_1 &\simdec u^\bullet  (w_m y w_{m-1}^\circ w _{m-2}) \dots w_1 \text{ by } \eqref{simdec-3} \\
& \simdec u^\bullet  w_m (w_{m-1} y w _{m-2}^\circ w_{m-3}) \dots w_1 \text{ by } \eqref{simdec-3}\\
& \simdec u^\bullet  w_m w_{m-1}  (w _{m-2} y w_{m-3}^\circ w_{m-4}) \dots w_1 \text{ by } \eqref{simdec-3}\\
& \simdec \dots\\
& \simdec u^\bullet w_m w_{m-1} \dots w_3 w_2 y w_1^\circ \text{ by } \eqref{simdec-3}.
\end{align*}}

\item[(iii)] Similar to (i), in this case $y \geq u > w_m$ and $w_m \leq w_{m-1} \leq w_j < y$, so
{\footnotesize\begin{align*}
u^\bullet y w_m^\circ w_{m-1} \dots w_1 &\simdec uw_m^\circ w_{m-1} \dots w_{j+1} y w_j^\bullet w_{j-1} \dots w_1 \text{ by (i) }\\
&\simdec u w_m^\circ w_{m-1} \dots w_{j+1} (y w_{j-1} w_j^\bullet w_{j-2})  \dots w_1 \text{ by }  \eqref{simdec-4} \\
&\simdec u w_m^\circ w_{m-1} \dots w_{j+1} y (w_{j-1} w_{j-2} w_j^\bullet w_{j-3})  \dots w_1 \text{ by }  \eqref{simdec-9} \\
&\simdec u w_m^\circ w_{m-1} \dots w_{j+1} y w_{j-1} ( w_{j-2}  w_{j-3}w_j^\bullet w_{j-4})  \dots w_1 \text{ by }  \eqref{simdec-9} \\
& \simdec \dots\\
& \simdec u w_m^\circ w_{m-1} \dots w_{j+1}y w_{j-1} \dots w_3 w_2w_1w_j^\bullet \text{ by }  \eqref{simdec-9}.
\end{align*}}
\item[(iv)] This case is similar to (iii), but now $u \leq w_m \leq w_{m-1} \leq w_j < y$. Therefore
{\footnotesize\begin{align*}
u^\bullet y w_m^\circ w_{m-1} \dots w_1 &\simdec u^\bullet w_m^\circ w_{m-1} \dots w_{j+1} y w_j^\circ w_{j-1} \dots w_1 \text{ by (ii) }\\
&\simdec u^\bullet w_m w_{m-1} \dots w_{j+1} (y w_{j-1} w_j^\circ w_{j-2})  \dots w_1 \text{ by }  \eqref{simdec-4} \\
&\simdec u^\bullet w_m w_{m-1} \dots w_{j+1} y (w_{j-1} w_{j-2} w_j^\circ w_{j-3})  \dots w_1 \text{ by }  \eqref{simdec-9} \\
&\simdec u^\bullet w_m w_{m-1} \dots w_{j+1} y w_{j-1} ( w_{j-2}  w_{j-3}w_j^\circ w_{j-4})  \dots w_1 \text{ by }  \eqref{simdec-9} \\
& \simdec \dots\\
& \simdec u^\bullet w_m w_{m-1} \dots w_{j+1}y w_{j-1} \dots w_3 w_2w_1w_j^\circ \text{ by }  \eqref{simdec-9}.
\end{align*}}
\item[(v)] In this case, depending on whether $u \leq w_m$ or $u > w_m$, we have
{\footnotesize\begin{align*}
(u^\bullet y w_m^\circ w_{m-1}) w_{m-2} \dots w_1 &\simdec u^\bullet (y w_{m-1} w_m^\circ w_{m-2})  \dots w_1 \text{ by } \eqref{simdec-4} \text{ or } \eqref{simdec-9}\\
& \simdec u^\bullet y (w_{m-1}  w_{m-2}w_m^\circ w_{m-3})  \dots w_1 \text{ by } \eqref{simdec-9} \\
& \simdec \dots\\
& \simdec u^\bullet w_{m-1} w_{m-2} \dots w_3 w_2w_1w_m^\circ \text{ by }  \eqref{simdec-9}.
\end{align*}}
\een
\end{proof}

The following example illustrate how Lemmas~\ref{simdec-incr} and \ref{simdec-decr} can be used to show that $\revrow(x^{\bullet} \toQQ T) \simdec x^{\bullet} \revrow(T)$ when $x \in \PP$ and $T$ is a one-row primed decomposition tableau. The above congruence relation is the key step in the proof of Proposition~\ref{simdec-prop} below.

\begin{example}
Suppose $x = 1$ and $T =  \ytab{ 4 & 2 & 2& 1' & 3 & 4 & 6} $. Then 
\[x \toQQ T = \ytab{ \none & 2' \\
4 & 3 & 2& 1 & 1 & 4 & 6}  %by Definition~\ref{decomp-insert-algo}, 
%\quand \revrow(x \toQQ T) = 64112342',
\] and we have
\begin{align*}
 x \hs\revrow(T) &= 16431'224\\
& \simdec 64131'224 &&\text{by Lemma~\ref{simdec-incr} (or Example~\ref{simdec-ex0})}\\
& \simdec 64112342' &&\text{by Lemma~\ref{simdec-decr}}\text{(iv)}\\
& = \revrow(x \toQQ T).
\end{align*}
Similarly, if $y = 4'$ then \[y \toQQ T = \ytab{ \none & 2' \\
4 & 4 & 2& 1' & 3 & 4 & 6}\] %by Definition~\ref{decomp-insert-algo}, 
%and so $\revrow(y \toQQ T) = 6431'2442'$. Then 
and we have
\begin{align*}
 y\hs \revrow(T) &= 4'6431'224\\
& \simdec 643'41'224 &&\text{by Lemma~\ref{simdec-incr}}\\
& \simdec 6431'2442' &&\text{by Lemma~\ref{simdec-decr}(iii)}\\
& = \revrow(y \toQQ T).
\end{align*}
%In the proof of Proposition~\ref{simdec-prop}, we will use all the subcases in Lemma~\ref{simdec-incr} and cases (i) to (v) in Lemma~\ref{simdec-decr} to show that $\revrow(x^{\bullet} \toQQ T) \simdec x^{\bullet} \revrow(T)$. 
%when $x \in \PP$ and $T$ is a one-row primed decomposition tableau.
\end{example}

The proposition and theorem below are the main results of this section.

\begin{proposition}\label{simdec-prop}
If $w$ is a primed word then $w \simdec \revrow(\Pqq(w))$.
\end{proposition}

\begin{proof}
By induction on the length of $w$, it suffices to show that 
\[\revrow(x^{\bullet} \toQQ T) \simdec x^{\bullet} \revrow(T) \] when $x \in \PP$ and $T$ is a primed decomposition tableau. By induction on the number of rows of $T$, it is enough to consider the case where $T$ has just one row. We will demonstrate the desired result in this case using Lemmas~\ref{simdec-incr} and \ref{simdec-decr}.

Assume $T$ is a one-row primed decomposition tableau, so that we can write
 \[T = \ytabd{ w_1 & w_2 & \dots & w_m^\circ & w_{m+1} & \dots & w_n}\] where $w_1w_2 \dots w_m \dots w_n$ is an unprimed hook word, and $w_m^\circ \in \{w_m', w_m\}$. Then $x^{\bullet} \revrow(T) = x^\bullet w_n w_{n-1} \dots w_m^\circ \dots w_2w_1$, and we have to show that this word is equivalent under $\simdec$ to $\revrow(x^{\bullet} \toQQ T)$.

Suppose $x > w_n$. Then 
\[x^{\bullet} \toQQ T = \ytabd{ w_1 & w_2 & \dots & w_m^\circ & w_{m+1} & \dots & w_n & x}\] by Definition~\ref{decomp-insert-algo} (notice the position of middle element does not change),  so
\[\ba \revrow(x^{\bullet} \toQQ T) &= xw_nw_{n-1} \dots w_m^\circ \dots w_2w_1 \\&\simdec x^\bullet w_n w_{n-1} \dots w_m^\circ \dots w_2w_1 = x^{\bullet} \revrow(T)\ea\] by \eqref{simdec-2}.

Suppose $m = n$ and $x \leq w_m$. Then  
\[ x^{\bullet} \toQQ T = \ytabb{ w_1 & w_2 & \dots & w_m & x^\bullet}\] by Definition~\ref{decomp-insert-algo} (notice the position of middle element has changed), so 
\[ \ba
\revrow(x^{\bullet} \toQQ T) &= x^\bullet w_m\dots w_2w_1 \\&\simdec x^\bullet w_m^\circ \dots w_2w_1 = x^{\bullet} \revrow(T)\ea\] by \eqref{simdec-1}.

From this point on we assume $m<n$ and $x \leq w_n$. When applying the algorithm in Definition~\ref{decomp-insert-algo} to $x^{\bullet} \toQQ T$, the middle element of the first row of $T$ moves to the right only when either (i) $x \leq w_m$ or (ii) $x$ bumps some $w_j$ in the increasing part, and this $w_j$ bumps $w_m^\circ$ to the next row. 

Suppose $x \leq w_m$. Then in the insertion process defining $x^{\bullet} \toQQ T$, the number $x$ bumps $w_{m+1}$ and $w_{m+1}$ bumps the leftmost entry $w_j$ with $w_j < w_{m+1}$ to the next row, where $j \in [m]$. The middle element moves to the right for any $j \in [m]$. Below, we examine the three possible subcases that can arise:
\bei
\item[(1)] First, if $2 \leq j<m$, then we have $w_{j-1} \geq w_{m+1} > w_{j}$ and
\[ x^{\bullet} \toQQ T = \ytabd{ \none & w_{j}^\circ \\ w_1 & \dots &w_{j-1} & w_{m+1} & w_{j+1} & \dots & w_m & x^\bullet & w_{m+2} & \dots & w_n},\] 
and one can check that
{\begin{align*}
x^{\bullet} \revrow(T) & = x^\bullet w_n w_{n-1} \dots w_m^\circ \dots w_2w_1 \\
& \simdec w_n w_{n-1} \dots w_{m+2} x^\bullet w_{m+1} w_m^\circ w_{m-1} \dots w_2 w_1 
%\text{ by Lemma~\ref{simdec-incr}}
\\
& \simdec w_n w_{n-1} \dots w_{m+2} x^{\bullet} w_m w_{m-1} \dots w_{j+1}w_{m+1} w_{j-1} \dots w_1 w_j^\circ % \text{ by Lem.~\ref{simdec-decr}(iv)}
\\
& = \revrow(x^{\bullet} \toQQ T),
\end{align*}}%
using Lemma~\ref{simdec-incr} and Lemma~\ref{simdec-decr}(iv) for two equivalences.

\item[(2)] If $j=1$ then
\[ x^{\bullet} \toQQ T = \ytabd{ \none & w_{1}^\circ \\ w_{m+1}& w_2  & \dots  & w_m & x^\bullet & w_{m+2} & \dots & w_n},\] 
and one can check that 
\begin{align*}
x^{\bullet} \revrow(T) & = x^\bullet w_n w_{n-1} \dots w_m^\circ \dots w_2w_1 \\
& \simdec w_n w_{n-1} \dots w_{m+2} x^\bullet w_{m+1} w_m^\circ w_{m-1} \dots w_2 w_1 % \text{ by Lemma~\ref{simdec-incr}}
\\
& \simdec w_n w_{n-1} \dots w_{m+2} x^{\bullet} w_m w_{m-1}  \dots w_2 w_{m+1} w_1^\circ 
%\text{ by Lemma~\ref{simdec-decr}(ii)}
\\
& = \revrow(x^{\bullet} \toQQ T),
\end{align*}
using Lemma~\ref{simdec-incr} and Lemma~\ref{simdec-decr}(ii) for the two equivalences.

\item[(3)] Finally, if $j = m$ then
\[ x^{\bullet} \toQQ T = \ytabd{ \none & w_{m}^\circ \\ w_1 & \dots  & w_{m-1} & w_{m+1} & x^\bullet & w_{m+2} & \dots & w_n},\] 
and one can check that
\begin{align*}
x^{\bullet} \revrow(T) & = x^\bullet w_n w_{n-1} \dots w_m^\circ \dots w_2w_1 \\
& \simdec w_n w_{n-1} \dots w_{m+2} x^\bullet w_{m+1} w_m^\circ w_{m-1} \dots w_2 w_1 
%\text{ by Lemma~\ref{simdec-incr}}
\\
& \simdec w_n w_{n-1} \dots w_{m+2} x^{\bullet}w_{m+1} w_{m-1}w_{m-2}  \dots w_1 w_m^\circ %\text{ by Lemma~\ref{simdec-decr}(v)}
\\
& = \revrow(x^{\bullet} \toQQ T)
\end{align*}
using Lemma~\ref{simdec-incr} and Lemma~\ref{simdec-decr}(v) for the two equivalences.
\eei

Now suppose $ w_{i-1} <x \leq w_{i}$ for some $i \in [m+1, n]$. Then  in the insertion process defining $x^{\bullet} \toQQ T$, the number $x$ bumps $w_i$ and $w_i$ bumps the leftmost entry $w_j$ with $w_j < w_{i}$ to the next row, where $j \in [m]$. We now have three more subcases according to whether $2 \leq j < m$, $j=1$, or $j=m$:
\bei
\item[(4)] If $2 \leq j < m$ then the position of the middle element is unchanged, and 
\[ x^{\bullet} \toQQ T = \ytabd{ \none & w_{j}^\bullet \\ w_1 & \dots &w_{j-1} & w_{i} & w_{j+1} & \dots & w_m^\circ & w_{m+1} &  \dots &w_{i-1} & x & w_{i+1} & \dots & w_n}.\] 
In this event, we check that
{\small\begin{align*}
x^{\bullet} \revrow(T) & = x^\bullet w_n w_{n-1} \dots w_m^\circ \dots w_2w_1 \\
& \simdec w_n w_{n-1} \dots w_{i+1} x w_{i-1} \dots w_{m+1}^{\bullet} w_i w_m^\circ w_{m-1} \dots w_2 w_1 %\text{ by Lemma~\ref{simdec-incr}}
\\
& \simdec w_n \dots w_{i+1} x w_{i-1} \dots w_{m+1}  w_m^\circ w_{m-1} \dots w_{j+1} w_i w_{j-1} \dots w_1 w_j^\bullet %\text{ by Lemma~\ref{simdec-decr}(iii)}
\\
& = \revrow(x^{\bullet} \toQQ T)
\end{align*}}%
using Lemma~\ref{simdec-incr} and Lemma~\ref{simdec-decr}(iii) for the two equivalences.

\item[(5)]
If $j = 1$ then again the position of the middle element in unchanged, and
\[ x^{\bullet} \toQQ T = \ytabd{ \none & w_{1}^\bullet \\ w_i & w_2 & \dots  & w_m^\circ & w_{m+1} &  \dots &w_{i-1} & x & w_{i+1} & \dots & w_n}.\] 
Now we check that
\begin{align*}
x^{\bullet} \revrow(T) & = x^\bullet w_n w_{n-1} \dots w_m^\circ \dots w_2w_1 \\
& \simdec w_n w_{n-1} \dots w_{i+1} x w_{i-1} \dots w_{m+1}^{\bullet} w_i w_m^\circ w_{m-1} \dots w_2 w_1 %\text{ by Lemma~\ref{simdec-incr}}
\\
& \simdec w_n \dots w_{i+1} x w_{i-1} \dots w_{m+1}  w_m^\circ w_{m-1} \dots w_2 w_i w_1^\bullet %\text{ by Lemma~\ref{simdec-decr}(i)}
\\
& = \revrow(x^{\bullet} \toQQ T)
\end{align*}
using Lemma~\ref{simdec-incr} and Lemma~\ref{simdec-decr}(i) for the two equivalences.

\item[(6)] Finally suppose $j = m$. Then $w_i$ replaces $w_m^\circ$ and since $w_i > w_{m+1}$, the position of the middle element changes to $w_{m+1}$. Therefore
\[ x^{\bullet} \toQQ T = \ytabd{ \none & w_{m}^\circ \\ w_1 & \dots  & w_{m+2} & w_i & w_{m+1}^\bullet & w_{m+2}&  \dots &w_{i-1} & x & w_{i+1} & \dots & w_n},\] 
and we check that
\begin{align*}
x^{\bullet} \revrow(T) & = x^\bullet w_n w_{n-1} \dots w_m^\circ \dots w_2w_1 \\
& \simdec w_n w_{n-1} \dots w_{i+1} x w_{i-1} \dots w_{m+1}^{\bullet} w_i w_m^\circ w_{m-1} \dots w_2 w_1 %\text{ by Lemma~\ref{simdec-incr}}
\\
& \simdec w_n \dots w_{i+1} x w_{i-1} \dots w_{m+1}^\bullet w_i  w_{m-1} w_{m-2}\dots w_1 w_m^\circ %\text{ by Lemma~\ref{simdec-decr}(v)}
\\
& = \revrow(x^{\bullet} \toQQ T)
\end{align*}
using Lemmas~\ref{simdec-incr} and \ref{simdec-decr}(v).
\eei
This completes our verification of the identity $\revrow(x^{\bullet} \toQQ T) \simdec x^{\bullet} \revrow(T) $ when $x \in \PP$ and $T$ is a one-row primed decomposition tableau.
\end{proof}

Our last theorem relates $\simdec$ and $\Pqq$ to the existence of $\qq_n$-isomorphisms.
\begin{theorem}
Suppose $v  $ and $w  $ are primed words with all letters at most $n$.
Let  $\cB$ and $\cC$ be the full subcrystals of $(\BB^+_n)^{\otimes \ell(v)}$ and  $(\BB^+_n)^{\otimes \ell(w)}$
that respectively contain $v$ and $w$.
Then the following properties are equivalent:
\ben
\item[(a)] It holds that $\Pqq(v) = \Pqq(w)$.
\item[(b)] It holds that $v\simdec w$.
\item[(c)] There exists a $\qq_n$-crystal isomorphism $\cB \to \cC$ sending $v\mapsto w$.
\een
\end{theorem}

\begin{proof}
We first check that (a) and (c) are equivalent.
Suppose $\Pqq(v) = \Pqq(w)$ and this shifted tableau has shape $\lambda$.
By Theorem~\ref{main-thm}, the operation $\Pqq$ defines $\qq_n$-crystal isomorphisms $\cB \to  \DDTab_n(\lambda)$ and $\cC \to  \DDTab_n(\lambda)$
sending $v$ and $w$ to the same element.
Composing the first isomorphism with the inverse of the second is a $\qq_n$-crystal isomorphism $\cB \to \cC$ sending $v\mapsto w$.
Therefore property (a) implies (c).

Conversely, suppose (c) holds so that there exists a $\qq_n$-crystal isomorphism $\phi : \cB \to \cC$ with $\phi(v) = w$.
 Because $\cB$ is a connected normal $\qq_n$-crystal,
 the results in Section~\ref{normal-sect} show that
 there is a unique strict partition $\lambda \in \NN^n$ with an isomorphism $\cB \to \DDTab_n(\lambda)$, which is also unique.
By Theorem~\ref{main-thm} the isomorphism $\cB \to \DDTab_n(\lambda)$ is just $\Pqq$ restricted to $\cB$.
 Since $\Pqq \circ  \phi$ is another crystal isomorphism from $\cB$
 to a primed decomposition tableau crystal $\DDTab_n(\mu)$ of some shape $\mu$,
 we must have $\lambda=\mu $ and $\Pqq(u) = \Pqq \circ \phi(u)$ for all $u \in \cB$.
In particular $\Pqq(v) = \Pqq(\phi(v)) = \Pqq(w)$.

To finish the proof of the theorem, it is now enough to show that property (a) implies (b) and property (b) implies (c).
The first implication is straightforward since if $\Pqq(v) = \Pqq(w)$
then we have $v\simdec \revrow(\Pqq(v)) = \revrow(\Pqq(w)) \simdec w$ by Proposition~\ref{simdec-prop}.
To discuss the second implication, 
we introduce some extra notation:
write $ v\equiv w$ if there exists a $\qq_n$-crystal isomorphism $\cB \to \cC$ sending $v\mapsto w$
as in property (c).
 
The following is an important observation regarding this notation.
Suppose $a$ and $b$ are primed words 
with all letters in $\{ 1'<1<\dots<n'<n\}$. Then $v\equiv w$
implies $avb \equiv awb$, since if $\phi :\cB \to \cC$ is an isomorphism sending $v \mapsto w$,
then
$\id_{(\BB_n^+)^{\ell(a)}} \otimes  \phi \otimes \id_{(\BB_n^+)^{\ell(b)}}$
restricts to the relevant isomorphism sending $avb\mapsto awb$.

We wish to show that if $v\simdec w$ then $v\equiv w$.
In view of the preceding paragraph,
it suffices to show that $v\equiv w$ in just the cases when $v$ and $w$ are the 2- or 4-letter primed words appearing in the relations 
in Definition~\ref{simdec}.
Since we already know that $v\equiv w$ if and only if $\Pqq(v) = \Pqq(w)$, 
we just need to check that 
the ten pairs of primed words in \eqref{simdec-1}-\eqref{simdec-10} have the same output under decomposition insertion.
This is a finite calculation, since it is sufficient to consider the cases when $\{a,b,c,d\}\subseteq \{1,2,3,4\}$.
For example, we have 
\[ \Pqq\(  a^\bullet b\) =  a^\bullet \toQQ b \toQQ \emptyset = \ytab{b & a^\bullet} = a^\bullet \toQQ b' \toQQ \emptyset = \Pqq\(a^\bullet b'\)
\]
when $a\leq b$, along with 
$  \Pqq\(  ba^\bullet \)  = \Pqq\(b'a^\bullet \) =  \ytab{a^\bullet & b}
$
when $a<b$, and 
\[\ba \Pqq\( a^\bullet bdc^\circ\) &= 
&a^\bullet \toQQ b \toQQ d \toQQ c^\circ \toQQ \emptyset 
\\&=
&a^\bullet \toQQ b \toQQ d \toQQ \ytab{c^\circ} 
\\&=
&a^\bullet \toQQ b \toQQ \ytab{c^\circ & d}
\\&=
&a^\bullet  \toQQ \ytab{\none & c^\circ \\ d & b}
\\
&=& \ytab{\none & c^\circ \\ d & b & a^\bullet}
\\
&=
&a^\bullet   \toQQ \ytab{c & b^\circ & d} 
\\
&=
&a^\bullet \toQQ d  \toQQ \ytab{c & b^\circ} 
\\
&=
&a^\bullet \toQQ d \toQQ b^\circ \toQQ \ytab{c} 
&\\
&=
&a^\bullet \toQQ d \toQQ b^\circ \toQQ c \toQQ \emptyset 
& & = \Pqq\( a^\bullet d b^\circ c\)
\ea
\]
for unprimed numbers $a\leq b \leq c <d$.
Similar calculations verify that $\Pqq(v) = \Pqq(w)$ when $v\simdec w$ are the 4-letter primed words in cases \eqref{simdec-4}-\eqref{simdec-10}.
This confirms that $v\equiv w$ when $v\simdec w$, as desired.
\end{proof}

Define $\DTab   := \bigsqcup_\lambda \DTab (\lambda)$ and $\DDTab   := \bigsqcup_\lambda \DDTab (\lambda)$ where both disjoint unions run over all strict partitions $\lambda$

\begin{corollary}
Both $\DTab $ and $\DDTab $ are monoids for the product \[T \ast U := \Pqq(\revrow(T) \revrow(U)).\]
\end{corollary}

\begin{proof}
As we have $\revrow(\Pqq(\revrow(T)\revrow(U))) \simdec \revrow(T) \revrow(U)$ it follows that  
$(T\bullet U) \bullet V = \Pqq(\revrow(T) \revrow(U) \revrow(V)) =T\bullet (U\bullet V)$.
Corollary~\ref{P-revrow-cor} implies that the empty tableau serves as the unit element.
\end{proof}

%\section*{Declarations}
%
%
%
%\subsection*{Ethical Approval}
%
%Not applicable.
%
%\subsection*{Funding} 
%
%This work was supported by Hong Kong RGC grants 16306120 and 16304122.
% 
%\subsection*{Availability of data and materials }
%
%Not applicable.

\printbibliography

\end{document}